\documentclass [a4paper, 12pt, reqno]{amsart}
\usepackage [latin1]{inputenc}
\usepackage {a4}
\usepackage{amscd}
\usepackage{epsfig}
\usepackage{amssymb}
\usepackage{amsmath}
\usepackage{amsthm}
\usepackage[T1]{fontenc}
\usepackage{ae,aecompl}
\usepackage[arrow, matrix, curve]{xy}

\usepackage{color}

\usepackage{geometry}
\newcommand{\R} {\ensuremath{\mathbb{R}}}

\newcommand{\C} {\ensuremath{\mathbb{C}}}

\newcommand{\OO}{\mathcal{O}}
\renewcommand{\o}[1]{\overline{#1}}

\newcommand{\h}[1]{\hat{#1}}

\newcommand{\dq}{\overline{\partial}}
\newcommand{\wt}[1]{\widetilde{#1}}
\DeclareMathOperator{\Reg}{Reg}
\DeclareMathOperator{\Sing}{Sing}

\DeclareMathOperator{\Dom}{Dom}

\newtheorem {satz} {Satz} [section]
\newtheorem {lem} [satz] {Lemma}

\newtheorem {defn} [satz] {Definition}

\newtheorem {conj} [satz] {Conjecture}

\newtheorem {thm} [satz] {Theorem}

\DeclareMathOperator{\supp}{supp}

\renewcommand{\theta}{\vartheta}

\title[$L^2$-theory for $\dq$ on compact spaces] 
{$L^2$-theory for the $\dq$-operator on compact complex spaces}

\author{J. Ruppenthal}

\address{Department of Mathematics, University of Wuppertal, Gau{\ss}str. 20, 42119 Wuppertal, Germany.}
\email{ruppenthal@uni-wuppertal.de}

\date{\today}

\subjclass[2000]{32J25, 32C35, 32W05}

\keywords{Cauchy-Riemann equations, $L^2$-theory, singular complex spaces.}

\begin{document}

\begin{abstract} 
Let $X$ be a singular Hermitian complex space of pure dimension $n$.
We use a resolution of singularities to give a smooth representation of the $L^2$-$\dq$-cohomology of $(n,q)$-forms on $X$.
The central tool is an $L^2$-resolution for the Grauert-Riemenschneider canonical sheaf $\mathcal{K}_X$.
As an application, we obtain a Grauert-Riemenschneider-type vanishing theorem for forms with values in almost positive line bundles.
If $X$ is a Gorenstein space with canonical singularities, then we get also an
$L^2$-representation of the flabby cohomology of the structure sheaf $\OO_X$.
To understand also the $L^2$-$\dq$-cohomology of $(0,q)$-forms on $X$,
we introduce a new kind of canonical sheaf, namely the canonical sheaf of square-integrable holomorphic
$n$-forms with some (Dirichlet) boundary condition at the singular set of $X$. 
If $X$ has only isolated singularities, then we use an $L^2$-resolution for that sheaf and a resolution
of singularities to give a smooth representation of the $L^2$-$\dq$-cohomology of $(0,q)$-forms.
\end{abstract}

\maketitle

~\\[-15mm]
\section{Introduction}
In the 1960s, the $L^2$-theory for the $\dq$-operator
has become an important, indispensable part of complex analysis
through the fundamental work of
H\"ormander on $L^2$-estimates and existence theorems for the $\dq$-operator (see \cite{Hoe1} and \cite{Hoe2})
and the related work of Andreotti and Vesentini (see \cite{AnVe}).
%\footnote{For the further development, we also recommend \cite{De}.}
One should also mention Kohn's solution of the $\dq$-Neumann problem (see \cite{Ko1}, \cite{Ko2} and also \cite{KoNi}),
which implies existence and regularity results for the $\dq$-complex, as well (see Chapter III.1 in \cite{FoKo}).
But whereas the theory is very well developed on complex manifolds,
it has been an open problem ever since to create an appropriate $L^2$-theory for the
$\dq$-operator on singular complex spaces.
We will give a partial answer to some aspects of that problem in the present paper.

When we consider the $\dq$-operator on singular complex spaces,
the first problem is to define an appropriate Dolbeault complex
in the presence of singularities. It turns out that it is very fruitful
to investigate the $\dq$-operator in the $L^2$-category (simply) on the 
complex manifold consisting of the regular points of a complex space.
One reason lies
in Goresky and MacPherson's notion of intersection (co-)homology (see \cite{GoMa1, GoMa2})
and the conjecture of Cheeger, Goresky and MacPherson,
which states that the $L^2$-deRham cohomology on the regular part of a projective variety $Y$
(with respect to the restriction of the Fubini-Study metric and the exterior derivate in the
sense of distributions) 
is naturally isomorphic to the intersection cohomology of middle perversity $IH^*(Y)$ of $Y$:

\begin{conj}\label{conj:cgm}{\bf (Cheeger-Goresky-MacPherson \cite{CGM})}\newline
Let $Y\subset \C\mathbb{P}^N$ be a projective variety. Then there is a natural isomorphism
\begin{eqnarray}\label{eq:cgm}
H^k_{(2)}(Y^*) \cong IH^k(Y).
\end{eqnarray}
\end{conj}

Here and throughout the paper, we write $Y^*$ for the regular part $Y\setminus \Sing Y$ of a singular space.
The early interest in this conjecture
was motivated in large parts by the hope that one could then
use the natural isomorphism and a classical Hodge decomposition $\oplus H^{p,q}$ for $H^k_{(2)}(Y^*)$ to
put a pure Hodge structure on the intersection cohomology of $Y$ (see \cite{PS1}, \cite{PS2} for more on this topic).

It is also interesting to have a look at the arithmetic genus of complex varieties.
When $M$ is a compact complex manifold, the arithmetic genus
$\chi(M) := \sum (-1)^q \dim H^{0,q}(M)$
is a birational invariant of $M$. The conjectured extension of the classical Hodge decomposition
to projective varieties led MacPherson also to ask whether the arithmetic genus $\chi(M)$
extends to a birational invariant of all projective varieties (see \cite{M}).
To formulate MacPherson's question slightly more generally,
we call a reduced and paracompact complex space $X$ Hermitian
if the regular part $X^*=X\setminus \Sing X$
carries a Hermitian metric which is locally the restriction of 
a Hermitian metric in some complex number space where $X$ is represented locally.
E.g. a projective variety is Hermitian with the the restriction of the Fubini-Study metric.

\begin{conj}\label{conj:pher}{\bf (MacPherson)}
If $X$ is a Hermitian compact complex space, then
\begin{eqnarray*}
\chi_{(2)}(X^*) := \sum (-1)^q \dim H^{0,q}_{(2)}(X^*) = \chi(M),
\end{eqnarray*}
where $\pi: M\rightarrow X$ is any resolution of singularities.
\end{conj}

Due to the incompleteness of the metric on $X^*=X\setminus \Sing X$,
one has to be very careful with the definition of Dolbeault cohomology groups $H^{0,q}_{(2)}$
for they depend on the choice of some kind of boundary condition for the $\dq$-operator.
To explain that more precisely, let $\dq_{cpt}$ be the $\dq$-operator acting on smooth $F$-valued forms 
with compact support away from $\Sing X$, where $F\rightarrow X^*=X\setminus\Sing X$
is a Hermitian holomorphic line bundle:
$$\dq_{cpt}:\  A^{p,q}_{cpt}(X^*,F) \rightarrow A^{p,q+1}_{cpt}(X^*,F).$$
We may consider $\dq_{cpt}$ as an operator acting on square-integrable forms:
$$\dq_{cpt}:\ \Dom \dq_{cpt} = A^{p,q}_{cpt}(X^*,F) \subset L^2_{p,q}(X^*,F) \rightarrow L^2_{p,q+1}(X^*,F).$$
This operator has various closed extensions. The two most important extensions
are the minimal closed extension, namely the closure of the graph of $\dq_{cpt}$ in
$L^2_{p,q}(X^*,F)\times L^2_{p,q+1}(X^*,F)$,
which we will denote by $\dq_{min}$, and the maximal closed extension, that is the $\dq$-operator
in the sense of distributions which we will denote by $\dq_{max}$.
It follows from the results below that they lead to different Dolbeault cohomology groups which we will
call $H^{p,q}_{min}(X^*,F)$ and $H^{p,q}_{max}(X^*,F)$, respectively.
This phenomenon occurs also in other singular situations (see \cite{BS}).

MacPherson's Conjecture \ref{conj:pher} has been settled for projective varieties
by Pardon and Stern \cite{PS1} for the arithmetic genus with respect to $\dq_{min}$, $\chi_{min}(X^*):=\sum (-1)^q \dim H^{0,q}_{min}(X^*)$.
Rather than comparing alternating sums, they realized that the groups $H^{0,q}_{min}$
themselves are birational invariants:

\begin{thm}\label{thm:ps1}{\bf (Pardon-Stern \cite{PS1})}
If $Y$ is a complex projective variety of pure dimension $n$ and $Y^*$ is given the Hermitian metric
induced by the embedding of $Y$ in projective space, then the groups $H^{0,q}_{min}(Y^*)$
are birational invariants of $Y$, and in fact, for $0\leq q\leq n$,
\begin{eqnarray*}
H^{0,q}_{min}(Y^*) \cong H^{0,q}(M),
\end{eqnarray*}
where $\pi: M\rightarrow Y$ is any resolution of singularities.
\end{thm}

For the $\dq$-operator in the sense of distributions, they claimed:

\begin{thm}\label{thm:ps2}{\bf (Pardon-Stern \cite{PS1})}
For $Y$ as in Theorem \ref{thm:ps1} with isolated singularities only, $\dim Y\leq 2$,
and $0\leq q\leq 2$,
\begin{eqnarray}\label{eq:ps2}
H^{0,q}_{max}(Y^*) \cong H^q(M,\OO(Z-|Z|)),
\end{eqnarray}
where $\pi: M\rightarrow Y$ is a resolution of singularities with only normal crossings,
and $Z$ the unreduced exceptional divisor $Z=\pi^{-1}(\Sing Y)$.
\end{thm}

It seems that the proof of Theorem \ref{thm:ps2} in \cite{PS1} is not complete (see Section \ref{ssec:ps2}).
A full proof was given recently by {\O}vrelid-Vassiliadou \cite{OvVa4}.
We may remark that {\O}vrelid-Vassiliadou make use of some of our results here.
We give another proof of Theorem \ref{thm:ps2} in Section \ref{ssec:ps2} by completing the argument of \cite{PS1}.

In the present paper, we generalize both,
Theorem \ref{thm:ps1} and Theorem \ref{thm:ps2}, to compact Hermitian complex spaces
of arbitrary dimension and forms with values in holomorphic line bundles. Our first main result is:

\begin{thm}\label{thm:A}
Let $X$ be a compact Hermitian complex space of pure dimension, 
$\pi: M\rightarrow X$ any resolution of singularities, and $L\rightarrow M$ a Hermitian line bundle which is locally semi-positive
with respect to the base space. Then the pull-back of forms under $\pi$ induces for all $0\leq q\leq n=\dim X$ a natural isomorphism
\begin{eqnarray}\label{eq:A}
\pi^*: H^{n,q}_{max}(X^*, \pi_* L) \overset{\cong}{\longrightarrow} H^{n,q}(M,L).
\end{eqnarray}
\end{thm}

Here, $L\rightarrow M$ is called semi-positive with respect to the base space if there is for each point $p\in X$
a neighborhood $U_p$ such that $L$ is semi-positive on $\pi^{-1}(U_p)$.
By $\pi_*L$ we denote the Hermitian holomorphic line bundle $(\pi|_{M\setminus E}^{-1})^* L$ over $X^*$,
where $E$ is the exceptional set of the resolution.

Of particular importance is the following situation. Let $F\rightarrow X$ be a Hermitian line bundle over $X$.
Then Theorem \ref{thm:A} applies to $L=\pi^* F$ and $\pi_* L = F|_{X^*}$ because the assumption of semi-positivity is trivially fulfilled.

Nevertheless, it is interesting to state Theorem \ref{thm:A} in this more general version
because in many situations one starts with a manifold $M$ and obtains $X$ as a reduction of $X$,
e.g. as a minimal model. Then we can deduce statements e.g. for a (globally) semi-positive line bundle $L\rightarrow M$.

To make the connection to Theorem \ref{thm:ps1}, we use $L^2$-Serre duality to deduce the dual version of Theorem \ref{thm:A}:

\begin{thm}\label{thm:Ad}
Under the assumptions of Theorem \ref{thm:A}, the push-forward of forms under $\pi$ 
induces for all $0\leq q\leq n$ a natural isomorphism
\begin{eqnarray}\label{eq:D}
\pi_*:  H^{0,q}(M,L^*) \overset{\cong}{\longrightarrow} H^{0,q}_{min}(X^*,\pi_* L^*).
\end{eqnarray}
\end{thm}

Here, $L^*$ is a Hermitian holomorphic line bundle that is locally semi-negative with respect to the base space.
Note that this settles MacPherson's Conjecture also for compact Hermitian complex spaces.

As a first application of Theorem \ref{thm:A}, we can give the following vanishing theorem of
Grauert-Riemenschneider-type:

\begin{thm}\label{thm:vanishing0}
Let $X$ be a pure-dimensional subvariety of a compact K\"ahler manifold, $\dim X=n$, $F\rightarrow X$
an almost positive holomorphic line bundle and $q>0$. Then:
\begin{eqnarray*}
H^q(X,\mathcal{K}_X(F)) = H^{n,q}_{max}(X^*,F)=H^{0,n-q}_{min}(X^*,F^*)=0.
\end{eqnarray*}
\end{thm}

Here, $\mathcal{K}_X(F)$ is the Grauert-Riemenschneider canonical sheaf of holomorphic $n$-forms with values in $F$.

A second application occurs in relation to the minimal model program.
If $X$ is a compact complex space as appearing naturally in the search for minimal models,
then we obtain statements also for the Grothendieck dualizing sheaf.
By use of Serre duality, this yields also an $L^2$-representation of the cohomology of the structure sheaf $\OO_X$,
or more generally any invertible sheaf on $X$:

\begin{thm}\label{thm:Ac}
Let $X$ be a compact Hermitian Gorenstein space of pure dimension $n$ with canonical singularities, $\pi: M\rightarrow X$
a resolution of singularities, $F\rightarrow X$ a holomorphic line bundle and
$\mathcal{F}\rightarrow X$ its associated sheaf of sections. 
Then there exist natural isomorphisms
\begin{eqnarray*}
H^q(X,\mathcal{\omega}_X\otimes \mathcal{F}) \cong H^{n,q}_{max}(X^*,F) \cong H^{n,q}(M,\pi^* F)\ ,\ 0\leq q\leq n,
\end{eqnarray*}
where $\omega_X$ is the Grothendieck dualizing sheaf. By duality we also have
\begin{eqnarray*}
H^{n-q}(X,\mathcal{F}^*) \cong H^{0,n-q}_{min}(X^*,F^*) \cong H^{0,n-q}(M,\pi^*F^*)\ ,\ 0\leq q\leq n.
\end{eqnarray*}
\end{thm}

Note that the last statement means that
$$H^q(X,\OO_X) \cong H^{0,q}_{min}(X^*) \cong H^{0,q}(M)$$
for any resolution of singularities and $0\leq q\leq n$.

\medskip

To approach the cohomology groups $H^{0,q}_{max}(X^*)$ in the spirit of Theorem \ref{thm:ps2},
we must develop a completely new approach because the techniques in \cite{PS1} are especially adopted to $\dim X\leq 2$.
More precisely, Pardon-Stern use Hsiang-Pati coordinates for a resolution of singularities of a normal complex surface (see \cite{HP}).

The key element of our approach here is a new kind of canonical sheaf on $X$ which we denote by $\mathcal{K}_X^s$.
It is the sheaf of germs of holomorphic square-integrable
$n$-forms which satisfy a Dirichlet boundary condition at the singular set $\Sing X$.
It comes as the kernel of the $\dq_s$-operator on square-integrable $(n,0)$-forms (see \eqref{eq:KXs} below).
The $\dq_s$-operator is a localized version of the $\dq_{min}$-operator (see Section \ref{ssec:ds} for the precise definitions). 
We denote by $\mathcal{F}^{p,q}$ the sheaves of germs of $L^2$-forms of degree $(p,q)$ in 
the domain of the $\dq_s$-operator. Then, by solving the $\dq_s$-equation for $(n,q)$-forms at isolated singularities,
we obtain:

\begin{thm}\label{thm:exactness2}
Let $X$ be a Hermitian complex space of pure dimension $n\geq 2$ with only isolated singularities. Then
\begin{eqnarray}\label{eq:exactness2}
0\rightarrow \mathcal{K}_X^s \hookrightarrow \mathcal{F}^{n,0} \overset{\dq_s}{\longrightarrow}
\mathcal{F}^{n,1} \overset{\dq_s}{\longrightarrow} \mathcal{F}^{n,2} \overset{\dq_s}{\longrightarrow} ... \longrightarrow \mathcal{F}^{n,n}
\rightarrow 0
\end{eqnarray}
is a fine resolution.
For an open set $U\subset X$, it follows that
$$H^q(U,\mathcal{K}_X^s) \cong H^q(\Gamma(U,\mathcal{F}^{n,*}))\ ,\ H^q_{cpt}(U,\mathcal{K}_X^s) \cong H^q(\Gamma_{cpt}(U,\mathcal{F}^{n,*})).$$
\end{thm}

If $X$ is compact, then this yields
\begin{eqnarray}\label{eq:intro21}
H^q(X,\mathcal{K}_X^s) \cong H^{n,q}_{min}(X^*) \cong H^{0,n-q}_{max}(X^*),
\end{eqnarray}
where the second isomorphism is by $L^2$-Serre duality. This gives some first understanding
of the $\dq_{max}$-cohomology of $(0,q)$-forms, showing e.g. that the groups are of finite dimension
because it is well known that $H^q(X,\mathcal{K}_X^s)$ is of finite dimension (we will see in a moment that $\mathcal{K}_X^s$ is actually coherent).

However, as above, we can give a smooth representation of \eqref{eq:intro21} in terms of a resolution of singularities.
Let us first study $\mathcal{K}_X^s$ closer.

\begin{thm}\label{thm:Ks}
Let $X$ be a Hermitian complex space of pure dimension with only isolated singularities.
%and $\pi: M\rightarrow X$ a resolution of singularities with only normal crossings. Then:
Then there exists a resolution of singularities $\pi: M\rightarrow X$ with only normal crossings 
and an effective divisor $D\geq Z-|Z|$ such that:
\begin{eqnarray}\label{eq:fix6}
\mathcal{K}_X^s \cong \pi_* \big( \mathcal{K}_M \otimes \OO(-D) \big),
\end{eqnarray}
where $\mathcal{K}_X^s$ is the canonical sheaf for the $\dq_s$-operator,
$\mathcal{K}_M$ is the usual canonical sheaf on $M$ and $Z=\pi^{-1}(\Sing X)$
the unreduced exceptional divisor.

If the exceptional set of the resolution $\pi: M\rightarrow X$ has only double self-intersections,
which is particularly the case if $\dim X=2$, %or $X$ can be resolved by a single blow-up,
then one can take $D=Z-|Z|$ in \eqref{eq:fix6}.
\end{thm}

By Grauert's direct image theorem, this yields particularly that $\mathcal{K}_X^s$ is a coherent analytic sheaf.

In contrast to the proof of Theorem \ref{thm:A} we are now in the more complicated situation
that the higher direct image sheaves 
$R^q\pi_* (\mathcal{K}_M\otimes\OO(-D))$, $q>0$,
do not vanish in general. Nevertheless, by a sophisticated use of the Leray spectral sequence,
we obtain our second main result:

\begin{thm}\label{thm:B}
Let $X$ be a Hermitian compact complex space of pure dimension with only isolated singularities. 
%and$\pi: M\rightarrow X$ a resolution of singularities with only normal crossings, $0\leq q\leq n=\dim X$.
Then there exists a resolution of singularities $\pi: M\rightarrow X$ with only normal crossings 
and an effective divisor $D\geq Z-|Z|$ such that
the following holds: Let $0\leq q\leq n=\dim X$.
%Then the pull-back of forms under $\pi$ induces a natural short exact sequence
Then there exists a short exact sequence
\begin{eqnarray}\label{eq:N}
0 \rightarrow H^{n,q}_{min}(X^*) \overset{i}{\longrightarrow} H^q\big(M,\mathcal{K}_M \otimes \OO(-D)\big) 
\longrightarrow \Gamma(X,\mathcal{R}^q) \rightarrow 0.
\end{eqnarray}
Here, $Z$ is the unreduced exceptional divisor $Z=\pi^{-1}(\Sing X)$, $\mathcal{R}^q$ is the direct image sheaf
$R^q\pi_* (\mathcal{K}_M\otimes\OO(-D))$ if $q>0$, and $\mathcal{R}^0\equiv 0$.
If the exceptional set of the resolution has only double self-intersections, then one can choose $D=Z-|Z|$
and the injection $i$ is directly induced by the pullback of forms under $\pi$.
\end{thm}

We make the connection to Theorem \ref{thm:ps2} by use of $L^2$-Serre duality:

\begin{thm}\label{thm:Bd}
Under the assumptions of Theorem \ref{thm:B}, 
%the push-forward of forms under $\pi$ induces for all $0\leq q\leq n$ a natural surjection
there is a surjection
\begin{eqnarray}\label{eq:DD}
p: H^{q}\big(M, \OO(D)\big) \longrightarrow H^{0,q}_{max}(X^*),
\end{eqnarray}
where the kernel is dual to $\Gamma(X,\mathcal{R}^q)$.
If we can choose $D=Z-|Z|$ then the surjection $p$ is induced by the push-forward of forms under $\pi$.
\end{thm}

Similar to Theorem \ref{thm:A}, one can prove Theorem \ref{thm:B} also for forms with values in a line bundle
$L\rightarrow M$ which is locally semi-positive with respect to the base space $X$.
We forgo that here as it would make the proof of Theorem \ref{thm:exactness2} considerably longer.
However, it is easy to see that the proofs of Theorem \ref{thm:B} and Theorem \ref{thm:Bd} apply
without additional difficulties to forms with values in a line bundle $F\rightarrow X$.
If $F$ is almost positive, then we get as in Theorem \ref{thm:vanishing0}:
$$H^q(X,\mathcal{K}_X^s(F)) = H^{n,q}_{min}(X^*,F) = H^{0,n-q}_{max}(X^*,F^*)=0\ ,\ q>0.$$

It is now interesting to ask under which circumstances the direct image sheaves $\mathcal{R}^q$ vanish
so that the maps $\pi^*$ and $\pi_*$ in \eqref{eq:N} and \eqref{eq:DD}, respectively, are isomorphisms.
{\O}vrelid-Vassiliadou showed in \cite{OvVa4} that $R^1\pi_* (\mathcal{K}_M\otimes \OO(|Z|-Z))=0$
if $\dim X=2$. Inserting that in Theorem \ref{thm:Bd} gives a proof of Theorem \ref{thm:ps2}.
However, also without that knowledge, surjectivity of $\pi_*$ in Theorem \ref{thm:Bd} is enough to fix the original proof
of Pardon-Stern. We will explain that in Section \ref{ssec:ps2}.

Besides, vanishing of the $\mathcal{R}^q$, $q>0$, follows from Takegoshi's vanishing theorem (see \cite{Ta}, Remark 2(a)) if the line bundle
associated to $\OO(-D)$ is locally semi-positive with respect to the base space $X$.
This is not true in general, but happens e.g. trivially if $Z=|Z|$ which is the case if the resolution is obtained
by a single blow-up of conical singularities (then we can take $D=Z-|Z|=\emptyset$). Let us summarize that:

\begin{thm}
Under the assumptions of Theorem \ref{thm:B}, assume that either $\dim X=2$ or that 
$X$ has only homogeneous singularities so that the resolution $\pi: M\rightarrow X$ is given by simple blow-ups
of the singularities.
Then the maps $i$ and $p$ in \eqref{eq:N} and \eqref{eq:DD}, respectively, are isomorphisms with $D=Z-|Z|$.
\end{thm}

The present paper is organized as follows. After collecting various preliminaries in Section \ref{sec:ht},
we prove the first main result, Theorem \ref{thm:A}, and its consequences in Section \ref{sec:thmA}.
In Section \ref{sec:canonical}, we develop the connection to Gorenstein spaces with canonical singularities.
Section \ref{sec:dolb} is devoted to the study of some $L^2$-$\dq$-results at isolated singularities
which we need to develop our theory of the $\dq_s$-operator and the canonical sheaf $\mathcal{K}_X^s$
in Section \ref{sec:psi}, which contains the proof of our second main result, Theorem \ref{thm:B}, and its consequences.
In three appendices, we prove some statements about the spectral sequence associated to a double complex,
compute a certain integral that is needed in the proof of Theorem \ref{thm:Ks},
and explain some statements on modifications of canonical sheaves.

\medskip
{\footnotesize
{\bf Acknowledgement.} The author thanks Nils {\O}vrelid for many very helpful discussions.
Particularly for pointing out the difficulty in the proof of Theorem \ref{thm:ps2} in \cite{PS1}
as well as the convergence of the integral in Appendix B.
The author thanks also Martin Sera and Matei Toma for clarifying discussions on the material contained in Appendix C
and moreover the unknown referees for their careful reading.
This research was supported by the Deutsche Forschungsgemeinschaft (DFG, German Research Foundation), 
grant RU 1474/2 within DFG's Emmy Noether Programme.}

%\newpage
\section{Preliminaries}\label{sec:ht}

\subsection{The weak $\dq$-operator $\dq_w$ and its $L^2$-complex}\label{sssec:dqw}

Let $(X,h)$ be a (singular) Hermitian complex space of pure dimension $n$,
$F\rightarrow X^*=X\setminus \Sing X$ a Hermitian holomorphic line bundle, $U\subset X$ an open subset.
On a singular space, it is fruitful
to consider forms that are square-integrable up to the singular set.
So, we will use the following concept of locally square-integrable forms with values in $F$:
\begin{eqnarray*}
L^{2,loc}_{p,q}(U,F):=\{f \in L^{2,loc}_{p,q}(U^*,F): f|_{K^*} \in L^{2}_{p,q}(K^*,F)\ \forall K\subset\subset U\}.
\end{eqnarray*}
It is easy to check that the presheaves given as
$$\mathcal{L}^{p,q}(U,F) := L^{2,loc}_{p,q}(U,F)$$
are already sheaves $\mathcal{L}^{p,q}(F)\rightarrow X$. On $L^{2,loc}_{p,q}(U,F)$, we denote by
$$\dq_w(U): L^{2,loc}_{p,q}(U,F) \rightarrow L^{2,loc}_{p,q+1}(U,F)$$
the $\dq$-operator in the sense of distributions on $U^*=U\setminus\Sing X$ which is closed and densely defined.
When there is no danger of confusion, we will simply write $\dq_w$ for $\dq_w(U)$.
%If $U$ is compact, then $\dq_w=\dq_{max}$.
The subscript refers to $\dq_w$ as an operator in a weak sense.
Since $\dq_w$ is a local operator, i.e.
$\dq_w(U)|_V = \dq_w(V)$
for open sets $V\subset U$,
we can define the presheaves of germs of forms in the domain of $\dq_w$,
$$\mathcal{C}^{p,q}(F):=\mathcal{L}^{p,q}(F)\cap \dq_w^{-1}\mathcal{L}^{p,q+1}(F),$$
given by
$\mathcal{C}^{p,q}(U,F) = \mathcal{L}^{p,q}(U,F) \cap\Dom\dq_w(U)$.
These are actually already sheaves
because the following is also clear: If $U=\bigcup U_\mu$ is a union of open sets, $f_\mu=f|_{U_\mu}$ and
$f_\mu \in \Dom \dq_w(U_\mu)$, then
$$f\in \Dom \dq_w(U)\ \ \  \mbox{ and }\ \ \  \big(\dq_w(U) f\big)|_{U_\mu} = \dq_w(U_\mu) f_\mu.$$
Moreover, it is easy to see that the sheaves $\mathcal{C}^{p,q}(F)$ admit partitions of unity,
and so we obtain sequences of fine sheaves
\begin{eqnarray}\label{eq:Cseq1}
\mathcal{C}^{p,0}(F) \overset{\dq_w}{\longrightarrow} \mathcal{C}^{p,1}(F) \overset{\dq_w}{\longrightarrow} \mathcal{C}^{p,2}(F) \overset{\dq_w}{\longrightarrow} ...
\end{eqnarray}
We use simply $\mathcal{C}^{p,q}$ to denote the sheaves of forms with values in the trivial line bundle.
We define
\begin{eqnarray}\label{defn:KX}
\mathcal{K}_X(F) := \ker \dq_w \subset \mathcal{C}^{n,0}(F).
\end{eqnarray}
We will see in the next section, when we deal with resolution of singularities, that $\mathcal{K}_X :=\ker\dq_w\subset \mathcal{C}^{n,0}$
is just the canonical sheaf of Grauert and Riemenschneider because the $L^2$-property of $(n,0)$-forms
remains invariant under modifications.

The $L^{2,loc}$-Dolbeault cohomology for forms with values in $F$ with respect to the $\dq_w$-operator 
on an open set $U\subset X$
is by definition the cohomology of the complex \eqref{eq:Cseq1} which is denoted by $H^q(\Gamma(U,\mathcal{C}^{p,*}(F)))$.
The cohomology with compact support is $H^q(\Gamma_{cpt}(U,\mathcal{C}^{p,*}(F)))$. Note that this is the cohomology
of forms with compact support in $U$, not with compact support in $U^*=U\setminus\Sing X$.

It is clearly interesting to study whether the sequence \eqref{eq:Cseq1} is exact,
which is well-known to be the case in regular points of $X$.
In singular points, the situation is quite complicated for forms of arbitrary degree and not completely understood.
However, we will show that the $\dq_w$-equation is locally solvable in the $L^2$-sense at arbitrary singularities for forms
of degree $(n,q)$, $q>0$, with values in a Hermitian holomorphic line bundle which is locally semi-positive
with respect to $X$.

\subsection{Resolution of singularities}\label{sssec:wresolution}

Throughout the paper, let $\pi: M \rightarrow X$
be a resolution of singularities (which exists due to Hironaka \cite{Hi}), i.e. a proper holomorphic surjection such that
\begin{eqnarray*}
\pi|_{M\setminus E}: M\setminus E \rightarrow X\setminus\Sing X
\end{eqnarray*}
is biholomorphic, where $E=|\pi^{-1}(\Sing X)|$ is the exceptional set.
We may assume that $E$ is a divisor with only normal crossings,
i.e. the irreducible components of $E$ are regular and meet complex transversely, but we do not need that for the moment.
Let $Z:=\pi^{-1}(\Sing X)$ be the unreduced exceptional divisor.
For the topic of desingularization, we refer to
\cite{AHL}, \cite{BiMi} and \cite{Ha}.
Let
$$\gamma:= \pi^* h$$
be the pullback of the Hermitian metric $h$ of $X$ to $M$.
$\gamma$ is positive semidefinite (a pseudo-metric) with degeneracy locus $E$.

We give $M$ the structure of a Hermitian manifold with a freely chosen (positive definite)
metric $\sigma$. Then $\gamma \lesssim \sigma$
and $\gamma \sim \sigma$ on compact subsets of $M\setminus E$.
For an open set $U\subset M$, we denote by $L^{p,q}_{\gamma}(U)$ and $L^{p,q}_{\sigma}(U)$
the spaces of square-integrable $(p,q)$-forms with respect to the (pseudo-)metrics $\gamma$ and $\sigma$,
respectively. 

Since $\sigma$ is positive definite and $\gamma$ is positive semi-definite,
there exists a continuous function $g\in C^0(M,\R)$ such that
\begin{eqnarray}\label{eq:l2dV}
dV_\gamma = g^2 dV_\sigma.
\end{eqnarray}
This yields $|g| |\omega|_\gamma  = |\omega|_\sigma$
if $\omega$ is an $(n,0)$-form, and
$|\omega|_\sigma \lesssim_U |g||\omega|_\gamma$
on $U\subset\subset M$ if $\omega$ is a $(n,q)$-form, $0\leq q\leq n$.\footnote{
This statement means that $|\omega|_\sigma/|\omega|_\gamma$ is locally bounded on $M$ for $(n,q)$-forms.}
So, for an $(n,q)$ form $\omega$ on $U\subset\subset M$:
\begin{eqnarray}\label{eq:l2est2}
\int_U |\omega|_\sigma^2 dV_\sigma \lesssim_U \int_U g^{2} |\omega|_\gamma^2 g^{-2} dV_\gamma = \int_U |\omega|^2_\gamma dV_\gamma.
\end{eqnarray}
Conversely,
$|g| |\eta|_\gamma \lesssim_U |\eta|_\sigma$
on $U\subset\subset M$ if $\eta$ is a $(0,q)$-form, $0\leq q\leq n$.\footnote{
For $(0,q)$-forms, $|\omega|_\gamma/|\omega|_\sigma$ is locally bounded.}
So, for a $(0,q)$ form $\eta$ on $U\subset\subset M$:
\begin{eqnarray}\label{eq:l2est}
\int_U |\eta|_\gamma^2 dV_\gamma \lesssim_U \int_U g^{-2} |\eta|_\sigma^2 g^2 dV_\sigma = \int_U |\eta|^2_\sigma dV_\sigma.
\end{eqnarray}
For open sets $U\subset\subset M$ and all $0\leq q\leq n$, we conclude the relations
\begin{eqnarray}\label{eq:l2est3}
L^{n,q}_{\gamma}(U) &\subset& L^{n,q}_{\sigma}(U),\\
%\end{eqnarray}
%and
%\begin{eqnarray}
\label{eq:l2est4}
L^{0,q}_{\sigma}(U) &\subset& L^{0,q}_{\gamma}(U).
\end{eqnarray}
If $L\rightarrow M$ is a Hermitian holomorphic line bundle over $M$,
we have:
\begin{eqnarray}\label{eq:l2est3b}
L^{n,q}_{\gamma}(U,L) &\subset& L^{n,q}_{\sigma}(U,L),\\
%\end{eqnarray}
%and
%\begin{eqnarray}
\label{eq:l2est4b}
L^{0,q}_{\sigma}(U,L) &\subset& L^{0,q}_{\gamma}(U,L).
\end{eqnarray}

For an open set $\Omega \subset X$, $\Omega^*=\Omega \setminus \Sing X$, $\wt{\Omega}:=\pi^{-1}(\Omega)$,
pullback of forms under $\pi$ gives the isometry
\begin{eqnarray}\label{eq:l2est5}
\pi^*: L^2_{p,q}(\Omega^*) \longrightarrow L^{p,q}_{\gamma}(\wt{\Omega}\setminus E) \cong L^{p,q}_{\gamma}(\wt{\Omega}),
\end{eqnarray}
where the last identification is by trivial extension of forms over the thin exceptional set $E$.
If $\pi_* L\rightarrow X\setminus \Sing X$ is the direct image, i.e. $\pi_* L = (\pi|_{M\setminus E}^{-1})^* L$,
then $\pi$ gives analogously the isometry
\begin{eqnarray}\label{eq:l2est5b}
\pi^*: L^2_{p,q}(\Omega^*,\pi_* L) \longrightarrow L^{p,q}_{\gamma}(\wt{\Omega}\setminus E,L) \cong L^{p,q}_{\gamma}(\wt{\Omega},L).
\end{eqnarray}

Combining \eqref{eq:l2est3b} with \eqref{eq:l2est5b},
we see that $\pi^*$ maps
\begin{eqnarray}\label{eq:morph1}
\pi^*: L^2_{n,q}(\Omega^*,\pi_* L) \rightarrow L^{n,q}_\sigma(\pi^{-1}(\Omega),L)
\end{eqnarray}
continuously if $\Omega\subset\subset X$ is a relatively compact open set.
We shall now show how \eqref{eq:morph1} induces the map
\begin{eqnarray}\label{eq:pi2}
\pi^*: H^{n,q}_{max}(X^*,\pi_* L) \rightarrow H^{n,q}(M,L)
\end{eqnarray}
from Theorem \ref{thm:A} (where $X$ is compact).

It makes sense to explain that from a slightly more general point of view.
For that, we need a suitable realization of the $L^2$-cohomology on $M$.
Let $\mathcal{L}^{p,q}_\sigma(L)$ be the sheaves of germs of forms on $M$ which are locally in $L^{p,q}_\sigma(L)$,
and we denote again by $\dq_w$ the $\dq$-operator in the sense of distributions on such forms
because there is no danger of confusion in what follows.
We can simply use the definitions from Section \ref{sssec:dqw} with the choice $X=M$ and $\Sing X=\emptyset$.
Again, we denote the sheaves of germs in the domain of $\dq_w$ by
$$\mathcal{C}^{p,q}_\sigma (L):= \mathcal{L}^{p,q}_\sigma (L)\cap \dq_w^{-1} \mathcal{L}^{p,q+1}_\sigma(L)$$
in the sense that
$\mathcal{C}^{p,q}_\sigma (U,L) = \mathcal{L}^{p,q}_\sigma (U,L) \cap \Dom \dq_w(U)$. It is well-known that
$$\mathcal{K}_M(L) := \ker \dq_w \subset \mathcal{C}^{n,0}_\sigma(L)$$
is the usual canonical sheaf on $M$ if $L$ is the trivial line bundle, and that
\begin{eqnarray}\label{eq:res}
0\rightarrow \mathcal{K}_M(L) \hookrightarrow \mathcal{C}^{n,0}_\sigma(L) \overset{\dq_w}{\longrightarrow} \mathcal{C}^{n,1}_\sigma(L)
\overset{\dq_w}{\longrightarrow} \mathcal{C}^{n,2}_\sigma(L) \longrightarrow ...
\end{eqnarray}
is a fine resolution so that 
$$H^q(U,\mathcal{K}_M(L)) \cong H^q(\Gamma(U,\mathcal{C}^{n,*}_\sigma(L)))\ ,
\ H^q_{cpt}(U,\mathcal{K}_M(L)) \cong H^q(\Gamma_{cpt}(U,\mathcal{C}^{n,*}_\sigma(L)))$$
on open sets $U\subset M$.

Now we can use \eqref{eq:morph1} to see that $\pi^*$ defines
a morphism of complexes 
\begin{eqnarray}\label{eq:morph2}
\pi^*: (\mathcal{C}^{n,*}(\pi_* L),\dq_w) \rightarrow (\pi_* (\mathcal{C}^{n,*}_\sigma(L)),\pi_* \dq_w).
\end{eqnarray}
Let $\Omega\subset X$ be an open set and let $f\in\mathcal{C}^{n,q}(\Omega,\pi_* L)$, $g\in \mathcal{C}^{n,q+1}(\Omega,\pi_* L)$
such that $\dq_w f=g$. By \eqref{eq:morph1}, it follows that $\pi^* f\in \mathcal{L}^{n,q}_\sigma(\pi^{-1}(\Omega),L)$
and $\pi^* g\in \mathcal{L}^{n,q+1}_\sigma(\pi^{-1}(\Omega),L)$ so that $\dq_w \pi^* f=\pi^* g$ on $\pi^{-1}(\Omega)\setminus E$.
But then the $L^2$-extension theorem \cite{Rp1}, Theorem 3.2, tells us that $\dq_w\pi^* f=\pi^*g$ on $\pi^{-1}(\Omega)$.
So $\pi^* f\in \mathcal{C}^{n,q}_\sigma(\pi^{-1}(\Omega),L)$, $\pi^* g\in \mathcal{C}^{n,q+1}_\sigma(\pi^{-1}(\Omega),L)$
and \eqref{eq:morph2} is in fact a morphism of complexes.
Including $\mathcal{K}_X(\pi_* L) =\ker\dq_w \subset \mathcal{C}^{n,0}(\pi_* L)$ and $\mathcal{K}_M(L)=\ker\dq_w \subset \mathcal{C}^{n,0}_\sigma(L)$,
we obtain the commutative diagram
\begin{eqnarray*}%\label{eq:diagram1}
\begin{xy}
  \xymatrix{
      0 \ar[r] & \mathcal{K}_X(\pi_* L) \ar[r] \ar[d]^{\pi^*}    &   \mathcal{C}^{n,0}(\pi_* L) \ar[r]^{\dq_w} \ar[d]^{\pi^*} & 
\mathcal{C}^{n,1}(\pi_* L) \ar[r]^{\dq_w} \ar[d]^{\pi^*} & \mathcal{C}^{n,2}(\pi_* L) \ar[r]%^{\dq_w} 
\ar[d]^{\pi^*} & ... \\
      0 \ar[r] & \pi_* (\mathcal{K}_M(L)) \ar[r] &  \pi_* (\mathcal{C}^{n,0}_\sigma(L)) \ar[r]^{\pi_* \dq_w}  & 
\pi_* (\mathcal{C}^{n,1}_\sigma(L)) \ar[r]^{\pi_*\dq_w} & 
\pi_* (\mathcal{C}^{n,2}_\sigma(L)) \ar[r]%^{\pi_* \dq_w} 
& ... }
\end{xy}
\end{eqnarray*}

It follows from commutativity of the diagram that $\pi^*$ induces a morphism on the cohomology of the complexes,
\begin{eqnarray}\label{eq:morph3}
\pi^*: H^q\big(\Gamma(\Omega,\mathcal{C}^{n,*}(\pi_* L))\big) \longrightarrow H^q\big(\Gamma(\pi^{-1}(\Omega),\mathcal{C}^{n,*}_\sigma(L))\big),
\end{eqnarray}
for any open set $\Omega\subset X$ and all $q\geq 0$.
If $X$ is compact and we choose $\Omega=X$, then the left hand side in \eqref{eq:morph3} is $H^{n,q}_{max}(X^*,\pi_* L)$
for $\dq_w(X^*)$ is the $\dq$-operator in the sense of distributions on $X^*$, i.e. $\dq_{max}$,
and the right hand side is just $H^{n,q}(M,L)$. This defines \eqref{eq:A} and \eqref{eq:pi2}, respectively.

\smallskip
We will now use Takegoshi's vanishing theorem \cite{Ta} to show that 
the lower line in the commutative diagram is exact if $L$ is locally semi-positive
with respect to the base space $X$.

Before, we shall mention another implication of the commutative diagram.
The vertical arrow on the left hand side is an isomorphism because $\mathcal{L}^{n,0}(\pi_* L) \cong \pi_* (\mathcal{L}^{n,0}_\sigma(L))$
and the $\dq_w$-equation extends over the exceptional set as described above (the $L^2$-extension \cite{Rp1}, Theorem 3.2).
So,
\begin{eqnarray}\label{KM}
\pi_* \big(\mathcal{K}_M(L)\big) \cong \mathcal{K}_X(\pi_* L).
\end{eqnarray}
Thus, $\mathcal{K}_X$ is in fact the canonical sheaf of Grauert--Riemenschneider as introduced in \cite{GrRie}.
We can use the direct image of the fine resolution \eqref{eq:res} of $\mathcal{K}_M(L)$
to express the cohomology of $\mathcal{K}_X(\pi_* L)$.
This follows by use of Takegoshi's vanishing theorem (see \cite{Ta}, Remark 2) which tells us that
the higher direct image sheaves of $\mathcal{K}_M(L)$ do vanish:
\begin{eqnarray}\label{eq:take1}
R^q\pi_* \big(\mathcal{K}_M(L)\big) =0,\ \ q>0,
\end{eqnarray}
if $L\rightarrow M$ is locally semi-positive with respect to the base space $X$.
%This statement also relies on $L^2$-methods. Recognizing that
%\begin{eqnarray}\label{eq:take2}
%\big(R^q\pi_* \mathcal{K}_M\big)_x = \lim_{\substack{\longrightarrow\\x\in U}} H^q(\pi^{-1}(U),\mathcal{K}_M),
%\end{eqnarray}
%Takegoshi proves his vanishing theorem by using a priori estimates
%to deduce that the $\dq$-equation is solvable in the $L^2$-sense for $(n,q)$-forms on $M$ in domains of the form $\pi^{-1}(U)$
%where $U$ is a small strongly pseudoconvex set in $X$.

Since \eqref{eq:res} is exact, \eqref{eq:take1} implies by use of the Leray spectral sequence that the lower line of the commutative diagram
is a fine resolution of $\pi_* (\mathcal{K}_M(L)) \cong \mathcal{K}_X(\pi_* L)$ (use also \eqref{KM}).
Note that the direct image of a fine sheaf under $\pi_*$ is again a fine sheaf.
So, we have proved:

\newpage
\begin{thm}\label{thm:resolution}
Let $X$ be a Hermitian complex space of pure dimension $n$, $\pi: M\rightarrow X$ a resolution of singularities with exceptional set $E$,
$\sigma$ a Hermitian metric on $M$, and $L\rightarrow M$ a Hermitian holomorphic line bundle which is locally semi-positive
with respect to the base space $X$. Then the $L^2$-complex $(\pi_* \mathcal{C}^{n,*}_\sigma(L),\pi_* \dq_w)$
is a fine resolution of the Grauert-Riemenschneider canonical sheaf with values in $\pi_* L$,
$$\mathcal{K}_X(\pi_* L) \cong \pi_* \mathcal{K}_M(L),$$
(where $\pi_* L= (\pi|_{M-E}^{-1})^* L|_{M-E}$). Thus:
\begin{eqnarray*}
H^q(\Omega,\mathcal{K}_X(\pi_* L)) \cong H^q(\Omega,\pi_*(\mathcal{K}_M(L))) \cong H^q(\pi^{-1}(\Omega),\mathcal{K}_M(L))
\end{eqnarray*}
for any open set $\Omega\subset X$ and all $q\geq 0$.
\end{thm}

Note that the assumption for the line bundle $L\rightarrow M$ is particularly satisfied in the following important situation.
Let $F\rightarrow X$ be a Hermitian holomorphic line bundle on $X$ and set $L:=\pi^* F$. Then $L \rightarrow M$ is clearly
locally semi-positive with respect to the base space $X$ because any holomorphic line bundle $F\rightarrow X$ is trivially
locally semi-positive. So, the statement of Theorem \ref{thm:resolution} holds with $F$ in place of $\pi_* L$ and
$\pi^* F$ in place of $L$, respectively.

\medskip
\subsection{Hermitian holomorphic line bundles}

Let $(M,\sigma)$ be a Hermitian complex manifold and $Z$ a divisor on $M$.
Let $\OO(Z)$ be the sheaf of germs of meromorphic functions $f$ such that $\mbox{div}(f)+Z\geq 0$.
We denote by $L_Z$ the associated holomorphic line bundle such that sections in $\OO(Z)$ correspond
to holomorphic sections in $L_Z$.
The constant function $f\equiv 1$ induces a meromorphic section $s_Z$ of $L_Z$
such that $\mbox{div}(s_Z)=Z$. One can then identify sections in $\OO(Z)$
with sections in $\OO(L_Z)$ by $g\mapsto g\otimes s_Z$,
and we denote the inverse mapping by $g\mapsto g\cdot s_Z^{-1}$.
If $D$ is an effective divisor, then $s_D$ is a holomorphic section of $L_D$ and
$\OO(-D) \subset \OO \subset \OO(D)$.

More generally, if $D$ is an effective divisor, then there is the natural inclusion $\OO(Z)\subset \OO(Z+D)$
which induces the inclusion $\OO(L_{Z})\subset \OO(L_{Z+D})$ given by
$g\mapsto (g\cdot s_{Z}^{-1})\otimes s_{Z+D}$.
For open sets $U\subset M$, this also induces the natural inclusion of smooth sections of vector bundles 
\begin{eqnarray}\label{eq:inclusion1}
\Gamma(U,L_{Z}) \subset \Gamma(U,L_{Z+D}).
\end{eqnarray}

We give each $L_Z$ the structure of a Hermitian holomorphic line bundle 
by choosing an arbitrary positive definite Hermitian metric $\langle\cdot,\cdot\rangle_{L_Z}$,
and we require that the dual bundle $L_Z^* = L_{-Z}$ carries the dual metric
$\langle\cdot,\cdot\rangle_{L_Z^*}=\langle\cdot,\cdot\rangle_{L_{-Z}}$.

If $U\subset\subset M$ is relatively compact and $\sigma$ any metric on $M$,
then \eqref{eq:inclusion1} induces the natural inclusion
\begin{eqnarray}\label{eq:inclusion2}
L^{p,q}_{\sigma}(U,L_Z) \subset L^{p,q}_{\sigma}(U,L_{Z+D})
\end{eqnarray}
for any effective divisor $D$. This does not depend on the metrics chosen on the line bundles
$L_Z$ and $L_{Z+D}$ because $U$ is relatively compact in $M$.
Note that \eqref{eq:inclusion2} is also valid with a positive semi-definite metric $\gamma$ in place of $\sigma$.

\smallskip
Now, let $F\rightarrow M$ be any Hermitian holomorphic line bundle.
As a connection on $F$ we use the Chern connection $D=D'+D''=D'+\dq$.

Our next purpose is to define the Hodge-$*$-operator for differential forms with values in $F$.
It is convenient to work with the conjugate-linear operator
$$\o{*}_\sigma \eta := *_\sigma \o{\eta},$$
where $*_\sigma$ is the usual Hodge-$*$-operator with respect to the metric $\sigma$.

Let
$\tau: F \rightarrow F^*$
be the canonical conjugate-linear bundle isomorphism of $F$ onto its dual bundle.
We can now define the conjugate-linear isomorphism
$$\o{*}_{F,\sigma}: \Lambda^{p,q} T^*M\otimes F \rightarrow \Lambda^{n-p,n-q} T^*M \otimes F^*$$
by setting $\o{*}_{F,\sigma}( \eta\otimes e) := \o{*}_\sigma \eta \otimes \tau(e)$.

This gives the following representation for
the inner product on $(p,q)$-forms with values in $F$:
\begin{eqnarray}\label{eq:ip1}
(\eta,\psi)_{F,\sigma} &=& \int_M \langle\eta,\psi\rangle_{F,\sigma} dV_\sigma = \int_M \eta\wedge \o{*}_{F,\sigma} \psi,\\
\|\eta\|_{F,\sigma} &=& \sqrt{(\eta,\eta)_{F,\sigma}}.\label{eq:ip2}
\end{eqnarray}
For later reference,
we remark that the operator $\o{*}_{F,\sigma}$ is denoted $\#$ in Demailly's introduction to Hodge theory \cite{De2}.
If $\gamma$ is a positive semi-definite metric with degeneracy locus contained in an exceptional set $E\subset M$,
then $\o{*}_\gamma$ and $\o{*}_{F,\gamma}$ are well-defined as above almost everywhere and \eqref{eq:ip1}, \eqref{eq:ip2} 
remain valid with $\gamma$ in place of $\sigma$.

\smallskip

Suppose that $\eta$, $\psi$ are smooth forms with values in $F$ and compact support in $M$,
$\eta$ of degree $(p,q-1)$ and $\psi$ of degree $(p,q)$.
So, $\eta\wedge \o{*}_{F,\sigma} \psi$ is a scalar valued $(n,n-1)$-form
and it is easy to compute:
\begin{eqnarray*}
d\big(\eta\wedge\o{*}_{F,\sigma} \psi\big) &=& \dq \big(\eta\wedge\o{*}_{F,\sigma} \psi\big)
= \dq\eta \wedge \o{*}_{F,\sigma}\psi + (-1)^{p+q-1} \eta\wedge \dq(\o{*}_{F,\sigma}\psi).
\end{eqnarray*}
It follows by Stokes' Theorem that
\begin{eqnarray*}
(\dq\eta,\psi)_{F,\sigma} &=& (-1)^{p+q}\int_M \eta\wedge\dq\o{*}_{F,\sigma} \psi
= -\int_M \eta\wedge \o{*}_{F,\sigma}\o{*}_{F^*,\sigma} \dq \o{*}_{F,\sigma}\psi\\
&=& (\eta, -\o{*}_{F^*,\sigma}\dq \o{*}_{F,\sigma}\psi)_{F,\sigma}.
\end{eqnarray*}
Thus, we note:

\begin{lem}\label{lem:theta}
The formal adjoint of the $\dq$-operator for forms with values in the holomorphic line bundle $F$
with respect to the $\|\cdot\|_{F,\sigma}$-norm is
\begin{eqnarray}\label{eq:theta}
\theta := - \o{*}_{F^*,\sigma} \dq \o{*}_{F,\sigma}.
\end{eqnarray}
\end{lem}

\bigskip

\subsection{$L^2$-Serre duality}\label{ssec:l2serre}

Let $(N,\sigma)$ be a Hermitian complex manifold of dimension $n$, $F\rightarrow M$ a Hermitian holomorphic line bundle,
and
$$\dq_{cpt}: A^{p,q}_{cpt}(N,F) \rightarrow A^{p,q+1}_{cpt}(N,F)$$
the $\dq$-operator on smooth $F$-valued forms with compact support in $N$.
Then we denote by
$$\dq_{max}: L^2_{p,q}(N,F) \rightarrow L^2_{p,q+1}(N,F)$$
the maximal and by
$$\dq_{min}: L^2_{p,q}(N,F) \rightarrow L^2_{p,q+1}(N,F)$$
the minimal closed Hilbert space extension of the operator $\dq_{cpt}$
as densely defined operator from $L^2_{p,q}(N)$ to $L^2_{p,q+1}(N)$.

For $F$-valued forms, 
let $H^{p,q}_{max}(N,F)$ be the $L^2$-Dolbeault cohomology on $N$ with respect
to the maximal closed extension $\dq_{max}$, i.e. the $\dq$-operator in the sense of distributions on $N$,
and $H^{p,q}_{min}(N,F)$ the $L^2$-Dolbeault cohomology with respect to the minimal closed
extension $\dq_{min}$.

We will now identify the Hilbert space adjoints $\dq_{max}^*$ and $\dq_{min}^*$ of $\dq_{max}$ and $\dq_{min}$.
Let $\theta$ be the formal adjoint of $\dq$ as computed in Lemma \ref{lem:theta},
and denote by $\theta_{cpt}$ its action on smooth $F$-valued forms compactly supported in $N$:
$$\theta_{cpt}: A^{p,q}_{cpt}(N,F) \rightarrow A^{p,q-1}_{cpt}(N,F).$$
This operator is graph closable as an operator $L^2_{p,q}(N,F)\rightarrow L^2_{p,q-1}(N,F)$,
and as for the $\dq$-operator, we denote by $\theta_{min}$ its minimal closed extension, i.e.
the closure of the graph, and by $\theta_{max}$ the maximal closed extension, that is the
$\theta$-operator in the sense of distributions with respect to compact subsets of $N$.

By \eqref{eq:theta}, it follows that
\begin{eqnarray*}
\theta_{min} = - \o{*}_{F^*,\sigma} \dq_{min} \o{*}_{F,\sigma}\ \ ,\ \ \theta_{max} = - \o{*}_{F^*,\sigma} \dq_{max} \o{*}_{F,\sigma}.
\end{eqnarray*}
By definition, $\dq_{max} = \theta_{cpt}^*$,
and it follows that
\begin{eqnarray}\label{eq:adjoint1}
\dq_{max}^* = \big(\theta_{cpt}^{*}\big)^* = \o{\theta_{cpt}} = \theta_{min} = -\o{*}_{F^*,\sigma}\dq_{min}\o{*}_{F,\sigma},
\end{eqnarray}
if we denote by $\o{\theta_{cpt}}$ also the closure of the graph of $\theta_{cpt}$.
Analogously, $\theta_{max}=\dq_{cpt}^*$ implies 
\begin{eqnarray}\label{eq:adjoint2}
\dq_{min}^* = \theta_{max} = - \o{*}_{F^*,\sigma} \dq_{max} \o{*}_{F,\sigma}.
\end{eqnarray}
As usual, this realization of the Hilbert space adjoints
yields harmonic representation of cohomology classes
and fundamental duality relations:

\begin{thm}\label{thm:duality}
Let $N$ be a Hermitian complex manifold of dimension $n$, $F\rightarrow N$ a Hermitian holomorphic line bundle,
and let $0\leq p,q\leq n$.
Assume that the $\dq$-operators in the sense of distributions
\begin{eqnarray}
\dq_{max}:&& L^2_{p,q-1}(N,F) \rightarrow L^2_{p,q}(N,F),\label{eq:dqmax1}\\
\dq_{max}:&& L^2_{p,q}(N,F) \rightarrow L^2_{p,q+1}(N,F)\label{eq:dqmax2}
\end{eqnarray}
both have closed range (with the usual assumptions for $q=0$ or $q=n$). 
Then there exists a non-degenerate pairing
$$\{\cdot,\cdot\}: H^{p,q}_{max}(N,F) \times H^{n-p,n-q}_{min}(N,F^*) \rightarrow \C$$
given by
$$\{\eta,\psi\}:=\int_N \eta\wedge\psi.$$
\end{thm}

\begin{proof}
The proof follows by standard arguments (representation of cohomology groups by harmonic representatives,
see Proposition 1.3 in \cite{PS1} and the Appendix in \cite{KK})
from the fact that the operators \eqref{eq:dqmax1}, \eqref{eq:dqmax2} have closed range exactly if their $L^2$-adjoints
have closed range. By \eqref{eq:adjoint1}, \eqref{eq:adjoint2} this is the case exactly if both the operators
\begin{eqnarray*}
\dq_{min}:&& L^2_{n-p,n-q}(N,F^*) \rightarrow L^2_{n-p,n-q+1}(N,F^*),\\
\dq_{min}:&& L^2_{n-p,n-q-1}(N,F^*) \rightarrow L^2_{n-p,n-q}(N,F^*)
\end{eqnarray*}
and their $L^2$-adjoints have closed range. So, $\o{*}_{F,\sigma}$ induces an isomorphism
\begin{eqnarray*}
H^{p,q}_{max}(N,F)\cong \ker \dq_{max} \cap \ker \dq_{max}^* \overset{\o{*}_{F,\sigma}}{\longrightarrow} 
\ker\dq_{min}\cap \ker\dq_{min}^* \cong H^{n-p,n-q}_{min}(N,F^*).
\end{eqnarray*}
\end{proof}

\bigskip

%\newpage
\section{$L^2$-Resolution of the Grauert-Riemenschneider canonical sheaf}\label{sec:thmA}

In this section, we prove Theorem \ref{thm:A} and derive dual statements by use of $L^2$-Serre duality.
The key ingredient here is a fine $L^2$-resolution for the Grauert-Riemenschneider canonical sheaf $\mathcal{K}_X$
with values in holomorphic line bundles.

\subsection{$L^2$-Solution of the $\dq_w$-equation on singular spaces}
We show that the $L^2$-$\dq_w$-complex $(\mathcal{C}^{n,*}(F))$ introduced in Section \ref{sssec:dqw}, \eqref{eq:Cseq1},
is a fine resolution for the Grauert-Riemenschneider canonical sheaf with values in certain line bundles:

\begin{thm}\label{thm:exactness1}
Let $X$ be a Hermitian complex space of pure dimension $n$,
and $F\rightarrow X^*=X\setminus \Sing X$ a Hermitian holomorphic line bundle which is locally semi-positive with respect to $X$,
i.e. for each point $x\in X$ there is a neighborhood $U_x\subset X$ such that $F$ is semi-positive on $U_x^*=U_x\setminus \Sing X$.
Then the complex
\begin{eqnarray}\label{eq:complex1}
0\rightarrow \mathcal{K}_X(F) \longrightarrow \mathcal{C}^{n,0}(F) \overset{\dq_w}{\longrightarrow}
\mathcal{C}^{n,1}(F) \overset{\dq_w}{\longrightarrow} \mathcal{C}^{n,2}(F) \overset{\dq_w}{\longrightarrow} ...
\end{eqnarray}
is exact, i.e. it is a fine resolution of $\mathcal{K}_X(F)$.
\end{thm}

Note that the assumption on $F$ is trivially fulfilled if $F$ extends to a holomorphic line bundle over $X$.
For the case of the trivial line bundle, $F=X\times\C$, Theorem \ref{thm:exactness1} is due to Pardon-Stern \cite{PS1}.

We will now prove Theorem \ref{thm:exactness1}. Our main tool is a vanishing theorem for complete K\"ahler manifolds
which we obtain by generalizing similar vanishing theorems 
of Donelly-Fefferman \cite{DF} and Ohsawa \cite{Oh2}:

\begin{thm}\label{thm:df2}
Let $N$ be a complete K\"ahler manifold of dimension $n$, whose K\"ahler metric $\omega$ is given by a potential
function $G: N\rightarrow \R$ as $\omega=i\partial\dq G$ such that $\langle\partial G,\partial G\rangle_\omega$ is bounded,
and let $(F,H)$ be a semi-positive Hermitian line bundle on $N$.
Then the $L^2$-$\dq$-cohomology in the sense of distributions with respect to $\omega$ for forms with values in $F$, 
$H^{n,q}_{max,\omega}(N,F)=0$ for $q>0$.
In fact, if $\langle\partial G,\partial G\rangle_\omega \leq B^2$, and $\phi$ is a $\dq$-closed $(n,q)$-form on $N$, $q>0$,
then there is a $(n,q-1)$-form $\nu$ such that $\dq\nu=\phi$ and $\|\nu\|_{\omega,H}\leq 4 B\|\phi\|_{\omega,H}$.
\end{thm}

\begin{proof}
Let $D=D'+D''=D'+\dq$ be the Chern connection on $F$.
It follows by \cite{Hoe2}, Lemma 4.1.1, that it is enough to show that
\begin{eqnarray}\label{eq:df1}
\|u\|_{\omega,H} \leq 4 |\partial G|_{\infty,\omega} \|(D'')^* u\|_{\omega,H}
\end{eqnarray}
for any $u\in \ker D''\cap \Dom (D'')^*\cap L^{n,q}_{\omega}(N,F)$ if $q\geq 1$.
The estimate \eqref{eq:df1} can be proved as follows:

For any differential form $\eta$, let $e(\eta)$ denote left multiplication by $\eta$,
$[\ ,\ ]$ the commutator with appropriate weight (i.e. $[S,T]=S\circ T- (-1)^{\deg S\deg T} T\circ S$),
and $^*$ the $L^2$-adjoint of an operator. With $\Lambda:=e(i\partial\dq G)^*$, it is well-known that
$[D'',\Lambda]=i(D')^*$ and $[e(\partial G),\Lambda]=i e(\dq G)^*$ (see \cite{De2}, 13.1, and \cite{W}, V.(3.22)).
Note that $\Lambda$, $e(\partial G)^*$ and $e(\dq G)^*$ are independent of the Hermitian vector bundle $(F,H)$.
Therefore:
\begin{eqnarray*}
[D'',e(\dq G)^*] &=& D''\circ e(\dq G)^* + e(\dq G)^*\circ D''\\
&=& -i D''\circ [e(\partial G),\Lambda] - i [e(\partial G),\Lambda]\circ D''\\
&=& [e(i\partial\dq G),\Lambda] + i e(\partial G) \circ[D'',\Lambda] + i [D'',\Lambda]\circ e(\partial G)\\
&=& [e(i\partial\dq G),\Lambda] - [e(\partial G),(D')^*].
\end{eqnarray*}
Reorganizing this, we have
$$[e(i\partial\dq G), \Lambda] =[D'',e(\dq G)^*] + [e(\partial G),(D')^*].$$
For a compactly supported smooth form $u$, that yields:
\begin{eqnarray*}
\big|([e(i\partial\dq G,\Lambda] u,u)_{\omega,H} \big|&=&
\big|(u, e(\dq G)\circ (D'')^* u)_{\omega,H} + (D''u, e(\dq G) u)_{\omega,H}\\
&+& (e(\partial G)\circ (D')^* u,u)_{\omega,H} + (e(\partial G) u, D'u)_{\omega,H} \big|\\
&\leq& |\partial G|_{\infty,\omega} \|u\|_{\omega,H}\\
&& \cdot\big(\|(D'')^*u\|_{\omega,H} + \|D''u\|_{\omega,H} + \|(D')^* u\|_{\omega,H} + \|D'u\|_{\omega,H}\big)\\
&\leq& 4 |\partial G|_{\infty,\omega} \|u\|_{\omega,H} \big(\|(D'')^*u\|_{\omega,H} + \|D''u\|_{\omega,H}\big),
\end{eqnarray*}
where the last step follows from the Bochner-Kodaira-Nakano identity
\begin{eqnarray*}
\|(D')^*u\|_{\omega,H}^2+\|D'u\|_{\omega,H}^2 &=& \|(D'')^*u\|_{\omega,H}^2 + \|D''u\|_{\omega,H}^2\\
&& - \int_N \langle [i\Theta(F),\Lambda]u,u\rangle_{\omega,H} dV_\omega
\end{eqnarray*}
(see \cite{De2}, Theorem 13.12) and the fact that $\langle [i\Theta(F),\Lambda]u,u\rangle_{\omega,H} \geq 0$
because $F$ is semi-positive (see \cite{De2}, 13.6).
In this step one sees that the theorem is also valid for $(0,q)$-forms with values in a semi-negative line bundle.
On the other hand,
$$([e(i\partial\dq G,\Lambda] u,u)_{\omega,H}=(p+q-n)\|u\|^2_{\omega,H}$$
if $u$ is a $(p,q)$-form (see \cite{W}, Proposition V.1.1(c)).
So, the estimate \eqref{eq:df1} is valid for smooth $(n,q)$-forms with compact support
if $q\geq 1$. As $(N,\omega)$ is complete, that implies \eqref{eq:df1} and the theorem is proved
(cf. \cite{De2}, Proposition 12.2).
\end{proof}

Let us now see how we can apply Theorem \ref{thm:df2} to prove Theorem \ref{thm:exactness1},
i.e. to solve the $\dq_w$-equation on $X$ locally for $(n,q)$-forms.

For the proof of Theorem \ref{thm:exactness1}, it is enough to consider the $\dq_w$-equation at singular points of $X$.
So, assume that a neighborhood of a point $p\in\Sing X$ is embedded holomorphically into $\C^L$, $L\gg n$,
such that $p=0\in\C^L$, and let $B_c$ be a ball of very small radius $c>0$ centered at the origin
such that $B_c\cap \Sing X$ is given as the common zero set of holomorphic functions $\{f_1, ..., f_m\}$ in $B_c$.
Following Pardon-Stern \cite{PS1}, we set
\begin{eqnarray}\label{eq:G}
G=-\log(c^2-|z|^2)
\end{eqnarray}
and
\begin{eqnarray}\label{eq:Gk}
G_k= -\log(c^2-|z|^2)-\frac{1}{k}\log\big(-\log\sum|f_j|^2\big)
\end{eqnarray}
for $z\in B_c$ and $k>1$, where $c$ is so small that $\sum|f_j|^2\ll 1$ on $B_c$.
Let $U:=X\cap B_c$. 
Then some computations yield (this is \cite{PS1}, Lemma 2.4):

\begin{lem}\label{lem:complete}
The metric $\omega_k:=i\partial\dq G_k$ on $U^*=U\setminus\Sing X$ is complete and decreases monotonically to $\omega:=i\partial\dq G$,
pointwise on $U^*$. $\langle \partial G_k,\partial G_k\rangle_{\omega_k}$ is bounded, independently of $k$,
where $\langle\cdot,\cdot\rangle_{\omega_k}$ denotes the pointwise metric on $1$-forms with respect to $\omega_k$.
\end{lem}

So, $U^*=U\setminus\Sing X$ carries a sequence of complete metrics $\omega_k$, decreasing to the incomplete metric $\omega$,
and we know by Theorem \ref{thm:df2} that we can solve the $\dq$-equation in the $L^2$-sense with respect to $\omega_k$
with a bound that does not depend on $k$. It is now essential to realize that $(n,q)$-forms do not only behave
well under a resolution of singularities, but also under such a decrease of the metric:

\begin{lem}\label{lem:vanishing}
Let $N$ be a complex manifold of dimension $n$ with a decreasing sequence of Hermitian metrics $\omega_k$, $k\geq 1$,
which converges pointwise to a Hermitian metric $\omega$. If $H^{n,q}_{max,\omega_k}(N,F)$ vanishes with an estimate that is
independent of $k$, then $H^{n,q}_{max,\omega}(N,F)$ vanishes with the same estimate.
\end{lem}

For the proof, we refer either to \cite{De}, Theorem 4.1, to \cite{Oh2}, Proposition 4.1,
or to \cite{PS1}, Lemma 2.3. 

Combining Theorem \ref{thm:df2}, Lemma \ref{lem:complete}
and Lemma \ref{lem:vanishing}, we see that the $\dq_w$-equation (in the sense of distributions) can be solved in the
$L^2$-sense for $F$-valued $(n,q)$-forms, $q\geq1$, on $U^*=U\setminus\Sing X$ with respect to the metric $\omega=i\partial\dq G$:
$$H^{n,q}_{max,\omega}(U^*,F)=0\ \ ,\ q\geq 1.$$
But $\omega$ is quasi-isometric to our original metric $h$ on any smaller subset $U'\subset\subset U$.
Hence, we see that in degree $(n,q)$, $q\geq 1$, the $\dq$-equation can be solved in the sense of distributions
in the $L^2$-category locally on a Hermitian complex space.
%\footnote{See also Remark 15.16 in Demailly's introduction to Hodge theory \cite{De2}.}

That completes the proof of Theorem \ref{thm:exactness1}.
Note that $\mathcal{K}_X(F) = \ker\dq_w \subset \mathcal{C}^{n,0}(F)$ by our Definition \eqref{defn:KX}
and that it is easy to see that the sheaves $\mathcal{C}^{n,q}(F)$ are fine.

\subsection{Proof of Theorem \ref{thm:A}}
It is now not hard to prove Theorem \ref{thm:A}. We just have to combine Theorem \ref{thm:resolution}
and Theorem \ref{thm:exactness1}.

To make that precise, let $X$ be compact Hermitian complex space of pure dimension $n$,
$\pi: M\rightarrow X$ a resolution of singularities, and $L\rightarrow M$ a Hermitian holomorphic line bundle
which is locally semi-positive with respect to the base space $X$.

Let $\sigma$ be any positive definite Hermitian metric on $M$.
Recall the morphism of complexes
\begin{eqnarray}\label{eq:morph21}
\pi^*: (\mathcal{C}^{n,*}(\pi_* L),\dq_w) \rightarrow (\pi_* (\mathcal{C}^{n,*}_\sigma(L)),\pi_* \dq_w)
\end{eqnarray}
that we set up in Section \ref{sssec:wresolution}, \eqref{eq:morph2}.
As we have seen (cf. \eqref{eq:morph3}), this morphism of complexes induces a morphism
on the cohomology of the complexes,
\begin{eqnarray}\label{eq:morph22}
\pi^*: H^q\big(\Gamma(\Omega,\mathcal{C}^{n,*}(\pi_* L))\big) \longrightarrow H^q\big(\Gamma(\pi^{-1}(\Omega),\mathcal{C}^{n,*}_\sigma(L))\big),
\end{eqnarray}
for any open set $\Omega\subset X$ and all $q\geq 0$.

But now Theorem \ref{thm:exactness1} and Theorem \ref{thm:resolution} tell us that both the complexes in \eqref{eq:morph21}
are a fine resolution for $\mathcal{K}_X(\pi_* L)=\pi_* \mathcal{K}_M(L)$. Thus, the map $\pi^*$ in \eqref{eq:morph22}
is an isomorphism for any open set $\Omega\subset X$ and all $q\geq 0$.

If $X$ is compact (as assumed) and we choose $\Omega=X$, then the left hand side in \eqref{eq:morph22} is 
by definition just $H^{n,q}_{max}(X^*,\pi_* L)$ and the right hand side is just $H^{n,q}(M,L)$.

\subsection{$L^2$-Duality -- Proof of Theorem \ref{thm:Ad}}

Under the assumptions of Theorem \ref{thm:A}, we just need to define a map on cohomology classes
\begin{eqnarray}\label{eq:dualmap1}
\pi_*: H^{0,n-q}(M,L^*) \longrightarrow H^{0,n-q}_{min}(X^*,\pi_* L^*)
\end{eqnarray}
which is dual to the isomorphism
\begin{eqnarray}\label{eq:dualmap2}
\pi^*: H^{n,q}_{max}(X^*,\pi_* L) \overset{\cong}{\longrightarrow} H^{n,q}(M,L).
\end{eqnarray}
Let $\sigma$ be a positive definite Hermitian metric on $M$, and $E$ the exceptional set of the resolution $\pi: M\rightarrow X$.
Then there is a natural (combined) map
\begin{eqnarray}\label{eq:dualmap3}
H^{0,n-q}_{min}(M\setminus E,L^*) \rightarrow H^{0,n-q}_{max}(M\setminus E,L^*) \overset{\cong}{\longrightarrow} H^{0,n-q}_{max}(M,L^*).
\end{eqnarray}
The map on the left-hand side of \eqref{eq:dualmap3} is well-defined as $\Dom(\dq_{min}) \subset \Dom(\dq_{max})$,
and the map on the right-hand side is well-defined and an isomorphism by the fact the the $\dq$-equation in the $L^2$-sense
extends over analytic sets (which we used before to define \eqref{eq:morph22}).
But is also known that the map on the left-hand side of \eqref{eq:dualmap3} is an isomorphism
(see e.g. \cite{PS1}, Proposition 1.12, and note that the line bundle $L^*$ does not matter for this statement).

So, a cohomology class $[\phi] \in H^{0,n-q}_{max}(M,L^*)$ has a representative $\phi\in\Dom \dq_{min}(M\setminus E,L^*)$
(which should be chosen to be identically zero if $[\phi]$ is zero).
It follows by \eqref{eq:l2est4} and the definition of the $\dq_{min}$-operator
that $(\pi|_{M\setminus E}^{-1})^* \phi \in \Dom\dq_{min}(X^*,\pi_* L^*)$,
and so we define the map \eqref{eq:dualmap1} by the assignment
\begin{eqnarray*}
\pi_* [\phi] := [(\pi|_{M\setminus E}^{-1})^* \phi] \in H^{0,n-q}_{min}(X^*,\pi_* L^*),
\end{eqnarray*}
where $\phi$ is such an appropriately chosen representative. So, the transformation law
$\int_{X^*} \eta\wedge\omega = \int_{M\setminus E} \pi^*\eta\wedge\pi^*\omega=\int_M\pi^*\eta\wedge\pi^*\omega$ induces the commutative diagram
\begin{eqnarray*}%\label{eq:diagram1}
\begin{xy}
  \xymatrix{
      H^{0,n-q}(M,L^*) \ar[r]^{\pi_*} \ar[d]^{\cong}  &  H^{0,n-q}_{min}(X^*,\pi_* L^*) \ar[d]^{\cong}  \\
      \big(H^{n,q}(M,L)\big)^*  \ar[r]^{(\pi^*)^*}_{\cong} &  \big(H^{n,q}_{max}(X^*,\pi_* L)\big)^* }
\end{xy}
\end{eqnarray*}
where it is well-known that the vertical arrow on the left hand-side is an isomorphism (Serre duality pairing),
and the vertical arrow on the right-hand side is an isomorphism by Theorem \ref{thm:duality} ($L^2$-Serre duality).

Note that the assumptions of Theorem \ref{thm:duality} are fulfilled by the following observation.
The cohomology groups $H^{n,q}_{max}(X^*,\pi_* L)$ are finite-dimensional for all $0\leq q\leq n$
because they are (by use of Theorem \ref{thm:A}) isomorphic to the groups $H^{n,q}(M,L)$
and $M$ is a compact complex manifold. But then all the $\dq$-operators
$$\dq_{max}: L^2_{n,q}(X^*,\pi_* L) \rightarrow L^2_{n,q+1}(X^*,\pi_* L)$$
under consideration have closed range by a standard argument following from Banach's open mapping theorem
(see \cite{HeLe}, Appendix 2.4).

\subsection{Vanishing Theorems}

\begin{defn}\label{defn:almost-positive}
A holomorphic line bundle $E$ on a (compact) complex space $X$ is called almost positive
if there is a Hermitian metric on $E$ whose curvature is semi-positive everywhere and positive
on some open set.
\end{defn}

This is a pretty useful concept in the context of modifications because it is bimeromorphically invariant.
We can use the following vanishing theorem of Nakano-Kodaira-Grauert-Riemenschneider type:

\begin{thm}\label{thm:NK}
Let $M$ be a compact K\"ahler manifold of dimension $n$ and $E$ an almost positive
line bundle on $M$. Then: $H^{n,q}(M,E)=0$ for $q>0$.
\end{thm}

\begin{proof}
The proof is just a variation of well-known arguments, so we shall be brief.
If $u$ is a $\dq$-harmonic $(n,q)$-form on $M$, then the Bochner-Kodaira-Nakano inequality (see \cite{De2}, 13.3) yields
$$\int_M \langle i\Theta(E)u,u\rangle\ dV_M \leq 0.$$
On the other hand (see \cite{De2}, 13.6),
$\langle i\Theta(E)u,u\rangle(x) \geq \gamma(x) |u(x)|$,
where $\gamma(x)$ is the smallest eigenvalue of the curvature of $E$ in the point $x$.
Thus, $u$ must vanish identically on the open set where $E$ is positive.
So, as a harmonic form, it vanishes everywhere.
\end{proof}

We can now show:

\begin{thm}\label{thm::vanishing1}
Let $X$ be a pure-dimensional subvariety of a compact K\"ahler manifold, $\dim X=n$, $F\rightarrow X$
an almost positive holomorphic line bundle and $q>0$. Then:
\begin{eqnarray*}
H^q(X,\mathcal{K}_X(F)) = H^{n,q}_{max}(X^*,F)=H^{0,n-q}_{min}(X^*,F^*)=0.
\end{eqnarray*}
\end{thm}

\begin{proof}
We may assume that $\pi: M\rightarrow X$ is an embedded resolution
obtained by finitely many blow-ups (i.e. monoidal transformations) along smooth centers, see \cite{BiMi}, Theorem 13.4.
So, $M$ can be interpreted as a submanifold in a finite product of K\"ahler manifolds,
inheriting a K\"ahler metric.
So, the statement follows directly by combining Theorem \ref{thm:A},
Theorem \ref{thm:Ad} and Theorem \ref{thm:NK}.
\end{proof}

\bigskip

%\newpage
\section{Gorenstein spaces with canonical singularities}\label{sec:canonical}

\subsection{Canonical sheaves on Gorenstein spaces}

For a complex space $X$ of pure dimension, we denote by $\omega_X$ the Grothendieck canonical sheaf,
i.e. the dualizing sheaf. If, for an open set $U$ in $X$, $X|_U \subset N$ is a local holomorphic embedding in a complex manifold $N$, then
$\omega_X \cong \mathcal{E}xt^r_{\OO_N}(\OO_X,\mathcal{K}_N)$, where $\mathcal{K}_N$ is the usual canonical sheaf on $N$
and $r$ is the codimension of $X$ in $N$.

Assume that $X$ is normal and $\iota: X\setminus\Sing X\hookrightarrow X$ the natural inclusion.
Then $\omega_X \cong \iota_*(\omega_{X\setminus\Sing X})$. This is \cite{GrRie}, Satz 3.1 (keep in mind that the singular set
of a normal space is of codimension $\geq$ 2). It is then clear that there is a natural inclusion 
of the Grauert-Riemenschneider canonical sheaf into the Grothendieck dualizing sheaf, $\mathcal{K}_X \subset \mathcal{\omega}_X$,
and this inclusion is strict in general (see \cite{GrRie}).

\begin{defn}\label{defn:Gorenstein}
Let $X$ be a normal complex space. Then $X$ is called Gorenstein if it is Cohen-Macaulay and
$\omega_X \cong \iota_*(\omega_{X\setminus\Sing X})$ is locally free (of rank one).
\end{defn}

We also recall what is meant by canonical singularities (see \cite{Kol}, Sect. 3):

\begin{defn}\label{defn:canonical}
Let $X$ be a compact Gorenstein space.
Then the dualizing sheaf $\omega_X$ is invertible and corresponds to a canonical Cartier divisor $K_X$.
We say that $X$ has canonical singularities if the following condition holds :
If $\pi: M \rightarrow X$ is any resolution of singularities and $K_M$
is the canonical divisor of $M$, so that we can write
\begin{eqnarray}\label{eq:gro3}
K_M = \pi^* K_X + \sum a_j E_j,
\end{eqnarray}
where the $E_j$ are the irreducible components of the exceptional divisor and the $a_j$ are rational coefficients.
Then $a_j\geq 0$ for all indices $j$.
\end{defn}

Thus, $X$ has canonical singularities precisely if $\pi^* \omega_X \subset \mathcal{K}_M$ for any resolution
of singularities $\pi: M\rightarrow X$ (as \eqref{eq:gro3} is equivalent to $\mathcal{K}_M=\pi^* \omega_X\otimes\OO(\sum a_j E_j)$ where
$\sum a_jE_j$ is an effective divisor).
So, note the following well-known essential fact about canonical sheaves on Gorenstein spaces
(which can be used even as a definition for canonical singularities):

\begin{thm}\label{thm:canonical}
Let $X$ be a compact Gorenstein space. Then the natural inclusion
$\mathcal{K}_X\subset \mathcal{\omega}_X$ induces an isomorphism $\mathcal{K}_X = \mathcal{\omega}_X$
exactly if $X$ has canonical singularities.
\end{thm}

\begin{proof}
Let $\pi: M\rightarrow X$ be any resolution of singularities.
Then $\mathcal{K}_X=\pi_* \mathcal{K}_M$ by \eqref{KM}.
Assume first that $\mathcal{K}_X = \omega_X$. Then:
$$\pi^* \omega_X = \pi^* \mathcal{K}_X = \pi^*\pi_* \mathcal{K}_M  \subset \mathcal{K}_M,$$
and so $X$ has canonical singularities.
Conversely, assume that $X$ has canonical singularities. We just have to show that $\omega_X \subset \mathcal{K}_X$.
As $X$ has canonical singularities, we know that $\pi^* \omega_X \subset \mathcal{K}_M$. But then:
$$\omega_X \subset \pi_* \pi^* \omega_X \subset \pi_* \mathcal{K}_M = \mathcal{K}_X.$$
\end{proof}

\subsection{Proof of Theorem \ref{thm:Ac}}

The first part of the statement is clear as we have 
$\mathcal{\omega}_X\otimes\mathcal{F} = \mathcal{K}_X\otimes\mathcal{F} \cong \mathcal{K}_X(F)$
by use of Theorem \ref{thm:canonical} and then we can apply Theorem \ref{thm:A}.
The second part follows by our $L^2$-duality Theorem \ref{thm:Ad} and Serre duality for singular spaces:

\begin{thm}{\bf (Ramis-Ruget \cite{RR})}
Let $X$ be a compact Cohen-Macaulay space of pure dimension $n$ and $\mathcal{F}\rightarrow X$ a locally free sheaf.
Then
$$H^q(X,\mathcal{F}) \cong H^{n-q}(X,\omega_X\otimes\mathcal{F}^*)\ ,\ 0\leq q\leq n.$$
\end{thm}

\medskip
\section{$L^2$-Dolbeault cohomology at isolated singularities}\label{sec:dolb}

As a preparation for the proof of Theorem \ref{thm:B}, we study different kinds
of $L^2$-Dolbeault cohomology in this section. Besides $\dq_{max}$ and $\dq_{min}$, we consider two other
closed extensions of the $\dq$-operator for which we prove some vanishing theorems.
Throughout this section, let
$X$ be a pure-dimensional complex analytic set in $\C^L$ of dimension $n\geq 2$ with an isolated singularity at the origin
such that $X$ carries the restriction of the Euclidean metric.
For small $r>0$, let $X^*=X-\{0\}$, $X_r= X\cap B_r(0)$ and $X_r^*=X_r-\{0\}$.

\subsection{$\dq$-operators with mixed boundary conditions}

Let
\begin{eqnarray*}\label{eq:l2}
L^{p,q}_{cpt} (X_r^*) &:=& \{ f\in L^{p,q}(X_r^*): \supp f\subset\subset X^*_r\},\\
L^{p,q}_{0}(X_r^*) &:=& \{f\in L^{p,q}(X_r^*): \supp f\cap \{0\} =\emptyset\},\\
L^{p,q}_b(X_r^*) &:=& \{f\in L^{p,q}(X_r^*): \supp f\cap bB_r(0)=\emptyset\},
\end{eqnarray*}
where the essential support $\supp f$ is taken in $X$. We may now consider the
operators (defined in the sense of distributions)
\begin{eqnarray*}
\dq_{cpt}: & L^{p,q}_{cpt}(X_r^*) \rightarrow L^{p,q+1}_{cpt}(X_r^*),\\
\dq_0: & L^{p,q}_{0}(X_r^*) \rightarrow L^{p,q+1}_{0}(X_r^*),\\
\dq_b: & L^{p,q}_{b}(X_r^*) \rightarrow L^{p,q+1}_{b}(X_r^*),
\end{eqnarray*}
and the formal adjoints $\theta_{cpt}=-\o{*}\dq_{cpt}\o{*}$, $\theta_0=-\o{*}\dq_0\o{*}$,
$\theta_b=-\o{*}\dq_b\o{*}$. All these operators are densely defined and graph closable because the
smooth forms with compact support are dense in each of the special $L^2$-spaces under consideration.
On the other hand, each of these $L^2$-spaces is dense in $L^{p,q}(X_r^*)$ resp. $L^{p,q+1}(X_r^*)$.
So, we can now consider
the closed extensions of $\dq_{cpt}$, $\dq_0$, $\dq_b$, respectively $\theta_{cpt}$, $\theta_0$ and $\theta_b$ 
as operators 
$L^{p,q}(X_r^*) \rightarrow L^{p,q+1}(X_r^*)$,
$L^{p,q+1}(X_r^*)\rightarrow L^{p,q}(X_r^*)$, respectively.

The maximal closed extensions of all these operators are the $\dq$ respectively the $\theta$-operator
in the sense of distributions $\dq_w$ and $\theta_w$, but we obtain some new minimal closed extensions.
First note that we already discussed the operators
$$\dq^{**}_{cpt} = \dq_{min}\ \ \mbox{ and }\ \ \theta_{cpt}^{**} = \theta_{min},$$
each coming with boundary conditions at both boundaries, $0$ and $bB_r(0)\cap X$.

We denote the other closures of the graphs as follows:
\begin{eqnarray*}
\dq_0^{**} =: \dq_{s,w}\ \ \mbox{ and } \ \ \theta_0^{**}=: \theta_{s,w},\\
\dq_b^{**} =: \dq_{w,s}\ \ \mbox{ and } \ \ \theta_b^{**}=: \theta_{w,s}.
\end{eqnarray*}
Here, the operators $\dq_{s,w}$ and $\theta_{s,w}$ have boundary conditions at $0$, whereas $\dq_{w,s}$ and $\theta_{w,s}$
come with boundary conditions at $bB_r(0)\cap X$.

For the $L^2$-adjoints, we obtain:

\begin{lem}\label{lem:sw}
$$\dq_{s,w}^* = \theta_{w,s}\ \ \mbox{ and } \ \ \dq_{w,s}^* = \theta_{s,w}.$$
\end{lem}

\begin{proof}
We will prove $\dq_{s,w}^*=\theta_{w,s}$, the second statement follows similarly.
It is enough to show that
$\Dom \dq_{s,w}^* = \Dom \theta_{w,s}$,
for if a form $f$ is in both domains, then it is clear that $\dq_{s,w}^* f =\theta_w f = \theta_{w,s}f$.
Let $f\in \Dom\theta_{w,s}$. So, $(f,\theta_{w,s})$ can be approximated by forms with support away from the boundary $b B_r(0)$.
Hence, partial integration is possible:
$$(f,\dq_{s,w} g)_{X_r^*} = (\theta_{w,s} f, g)_{X_r^*} \ \forall g\in\Dom\dq_{s,w},$$
since $(g,\dq_{s,w}g)$ can be approximated by forms with support away from the origin.

Conversely, let $f\in \Dom\dq_{s,w}^*$.
Let $\chi \in C^\infty(X)$ be a cut off function such that $\chi\equiv 0$
in a neighborhood of the origin and $\chi\equiv 1$ outside $B_{r/2}(0)\cap X$.
Then
$$\chi f \in \Dom\dq_w^* = \Dom \theta_s \subset \Dom\theta_{w,s}.$$
On the other hand, it is clear that $(1-\chi)f \in L_b^{p,q}(X_r^*)$
is in $\Dom\theta_{w,s}$. Hence:
$$f=\chi f + (1-\chi) f\ \in \Dom\theta_{w,s}.$$
\end{proof}

Let $H^{p,q}_{s,w}(X_r^*)$ and $H^{p,q}_{w,s}(X_r^*)$ be the $L^2$-Dolbeault cohomology groups with respect to the
$\dq_{s,w}$ and the $\dq_{w,s}$-operator, respectively. If we define the spaces of harmonic forms as
\begin{eqnarray*}
\mathcal{H}^{p,q}_{s,w}(X_r^*) &=& \ker \dq_{s,w}\cap \ker \dq_{s,w}^*,\\
\mathcal{H}^{p,q}_{w,s}(X_r^*) &=& \ker \dq_{w,s}\cap \ker \dq_{w,s}^*,
\end{eqnarray*}
then we obtain as in Theorem \ref{thm:duality} the duality:
\begin{eqnarray}\label{eq:sw}
\mathcal{H}^{p,q}_{s,w}(X_r^*) \cong \mathcal{H}^{n-p,n-q}_{w,s}(X_r^*),
\end{eqnarray}
where the isomorphism is given by application of the $\o{*}$-operator.
\eqref{eq:sw} extends to an isomorphism of the associated
cohomology groups if one of these groups is finite-dimensional.
So, we can include a dual statement in all the following results.

\subsection{$L^2$-vanishing theorems at isolated singularities}

\begin{lem}\label{lem:dolb1}
$$H^{0,n-q}_{min}(X_r^*) \cong H^{n,q}_{max}(X_r^*)=0\ ,\  q\geq 1.$$
\end{lem}

\begin{proof}
Let $f\in L^{n,q}(X_r^*)$ such that $\dq_wf=0$. Then it follows by Theorem \ref{thm:exactness1}
that there exists a local $L^2$-solution of the $\dq_w$-equation at the origin.
Since $X_r$ has strongly pseudoconvex boundary in $X$, this solution can be extended to
$X_r^*$ by standard methods (see e.g. section VIII.4 in \cite{LM}). 
Another proof of this fact is in \cite{FOV2}, Theorem 1.2.
That shows $H^{n,q}_{max}(X_r^*)=0$, and $H^{0,n-q}_{min}(X_r^*)=0$ follows by duality (Theorem \ref{thm:duality}).
\end{proof}

We will now show that also $H^{0,q}_{w,s}(X_r^*)=0$ for $0\leq q\leq n-1$.
For the proof of this statement, we need two vanishing results of Forn{\ae}ss-{\O}vrelid-Vassiliadou.
The first one is Proposition 3.1 in \cite{FOV2}:

\begin{lem}\label{lem:dolb2}
Let $p+q<n$, $q>0$, and $0<r_0<r$. Let $f\in L^{p,q}(X_r^*)$ such that $\dq_w f=0$ and $\supp f\subset \o{X_{r_0}}$.
Then there exists $u\in L^{p,q-1}(X_r^*)$ with the following properties: $\dq_w u =f$ and $u$ can be extended trivially by $0$ to $X^*$
s.t. $\dq_w u=f$ on $X^*$.
\end{lem}

The second statement is the fact that 
the $\dq_w$-equation can be solved at isolated singularities in certain degrees with a gain of regularity
that is enough to ensure that the solution is in the domain of $\dq_{s,w}$:

\begin{lem}\label{lem:fov}
Let $p+q\geq n+2$, $q>0$. Given $f\in L^{p,q}(X_r^*)$, $\dq_w f=0$ in $X_r^*$, 
there exists $u$ satisfying $\dq_w u=f$ in $X_r^*$ with
\begin{eqnarray*}\label{eq:gospel25}
\int_{X^*_{r_0}} \|z\|^{-2}\left(-\log\|z\|^2\right)^{-2}|u|^2 dV \leq C(r_0) \int_{X_r^*} |f|^2 dV
\end{eqnarray*}
where $0<r_0<r$ and $C(r_0)$ is a positive constant that depends on $r_0$.
$u$ can be approximated by a sequence of smooth forms $u_k$ with compact support
away from the origin such that $u_k\rightarrow u$, $\dq u_k\rightarrow f$ in $L^2_{p,*}(X_{r_0}^*)$,
thus $\dq_{s,w} u=f$.
\end{lem}

\begin{proof}
The first statement is just Theorem 1.2 in \cite{FOV2}.
The second statement, $\dq_{s,w} u=f$, follows as the last statement of Theorem 1.1 in \cite{FOV2},
or by the cut-off procedure of our Lemma \ref{lem:fix1} below. Forn{\ae}ss-{\O}vrelid-Vassiliadou use the same cut-off
procedure for the last statement of their Theorem 1.1.
\end{proof}

Combining Lemma \ref{lem:dolb1} with Lemma \ref{lem:dolb2} and Lemma \ref{lem:fov}, we obtain:

\begin{lem}\label{lem:dolb3}
Let $c>0$ small enough. Then:
$$H^{n,n-q}_{s,w}(X_c^*) \cong H^{0,q}_{w,s}(X_c^*) = 0\ ,\ 0\leq q\leq n-1.$$
\end{lem}

\begin{proof}
We distinguish two overlapping cases. First, let $1\leq q\leq n-1$
and $f\in L^{0,q}(X_c^*)$ such that $\dq_{w,s} f=0$.
Let $f^0$ be the trivial extension of $f$ to $X^*$.
Then $\dq_w f^0=0$ on $X^*$ because $f$ can be approximated in the graph norm
by a sequence in $L^{0,q}_b(X_c^*)$ and this sequence approximates $f^0$ as well.

Choose $r>r_0>c>0$ adequately,
and let $u$ be the solution to $\dq_w u = f^0$
from Lemma \ref{lem:dolb2}. Now then, let $\chi\in C^\infty(X)$
be a cut-off function such that $\chi\equiv 1$ in a neighborhood of the origin
and $\chi\equiv 0$ outside $X\cap B_{c/2}(0)$.
Consider
$$f' := f - \dq_w( \chi u),$$
which is $\dq_w$-closed, vanishes in a neighborhood of the origin
and equals $f$ outside the ball $B_{c/2}(0)$, especially close to the
boundary $X\cap bB_c(0)$. But then $\dq_{min} f'=0$
for the $\dq_{min}$ operator on $X_c^*$.
By Lemma \ref{lem:dolb1}, there exists $g\in L^{0,q-1}(X_c^*)$
such that $\dq_{min} g = f'$. It is then clear that also
$\dq_{w,s} g = f'$.
On the other hand, $\chi u$ is identically zero outside $B_{c/2}(0)$,
yielding
$$\dq_{w,s} (\chi u) = \dq_w (\chi u)$$
on $X_c^*$. Hence $g':= g + \chi u$
is the desired solution to the $\dq_{w,s}$-equation:
$$\dq_{w,s} g' = \dq_{w,s} g + \dq_{w,s}(\chi u) = f' + \dq_w(\chi u) = f.$$
Thus $H^{0,q}_{w,s}(X_c^*)=0$ for $1\leq q\leq n-1$ and $H^{n,n-q}_{s,w}(X_c^*)=0$ by duality (see \eqref{eq:sw}).

The missing case $q=0$ could be obtained by using a resolution of singularities to show
that a $\dq_{w,s}$-holomorphic function corresponds to a vector-valued holomorphic function with compact support on the resolution,
so it must vanish identically. We choose another approach.

For the second case, let $0\leq q \leq n-2$.
Then Lemma \ref{lem:fov} yields by use of $\Dom \dq_{s,w} \subset \Dom \dq_w$
that $H^{n,n-q}_{s,w}(X_c^*)=0$. Duality gives $H^{0,q}_{w,s}(X_c^*)=0$.
\end{proof}

\medskip
\section{A new canonical sheaf and its $L^2$-resolution}\label{sec:psi}

\subsection{The strong $\dq$-operator $\dq_s$ and its $L^2$-complex}\label{ssec:ds}

We introduce now a suitable local realization of a minimal version of the $\dq$-operator.
This is the $\dq$-operator with a Dirichlet boundary condition at the singular set $\Sing X$ of $X$.
As in Section \ref{sssec:dqw},
let $X$ be a (singular) Hermitian complex space of pure dimension $n$,
$H\rightarrow X^*=X\setminus \Sing X$ a Hermitian holomorphic line bundle, $U\subset X$ an open subset.
Let
$$\dq_s(U): \mathcal{L}^{p,q}(U,H) \rightarrow \mathcal{L}^{p,q+1}(U,H)$$
be defined as follows. We say that $f\in\Dom\dq_w(U)$ is in the domain of $\dq_s(U)$ if there exists a
sequence of forms $\{f_j\}_j \subset \Dom\dq_w(U)=\mathcal{C}^{p,q}(U,H) \subset \mathcal{L}^{p,q}(U,H)$ with essential support away from the singular set,
$\supp f_j \cap \Sing X = \emptyset$, such that
\begin{eqnarray}\label{eq:ds1}
f_j \rightarrow f &\mbox{ in }& L^{p,q}(K\setminus\Sing X,H),\\
\dq_w f_j \rightarrow \dq_w f &\mbox{ in }& L^{p,q+1}(K\setminus\Sing X,H)\label{eq:ds2}
\end{eqnarray}
for each compact subset $K\subset\subset U$. 
Again, we write simply $\dq_s$ for $\dq_s(U)$ 
if there is no danger of confusion.
The subscript refers to $\dq_s$ as an extension in a strong sense.

Note that we can assume without loss of generality
(by use of cut-off functions and smoothing with Dirac sequences) that the forms $f_j$ are smooth with compact support in $U\setminus\Sing X$.
In particular, if $X$ is compact, then 
\begin{eqnarray}\label{eq:smin}
\dq_s(X)=\dq_{min}: L^2_{p,q}(X^*,H) \rightarrow L^2_{p,q+1}(X^*,H).
\end{eqnarray}
It is now clear that $\dq_s(U)|_W = \dq_s(W)$
for open sets $W\subset U$, and we can define the presheaves of germs of forms in the domain of $\dq_s$,
$$\mathcal{F}^{p,q}(H):=\mathcal{L}^{p,q}(H)\cap \dq_s^{-1}\mathcal{L}^{p,q+1}(H),$$
given by $\mathcal{F}^{p,q}(U,H) = \mathcal{L}^{p,q}(U,H) \cap \Dom \dq_s(U)$.

Here, we shall check a bit more carefully that these are already sheaves:
Let $U=\bigcup U_\mu$ 
be a union of open sets, $f\in \mathcal{L}^{p,q}(U,H)$ and $f_\mu=f|_{U_\mu} \in \Dom\dq_s(U_\mu)$
for all $\mu$. We claim that $f\in\Dom\dq_s(U)$.
To see this, we can assume (by taking a refinement if necessary)
that the open cover $\mathcal{U}:=\{U_\mu\}$ is locally finite,
and choose a partition of unity $\{\varphi_\mu\}$ for $\mathcal{U}$.
On $U_\mu$ choose a sequence $\{f^\mu_j\}\subset \mathcal{L}^{p,q}(U_\mu,H)$ as in \eqref{eq:ds1}, \eqref{eq:ds2},
and consider
$$f_j := \sum_{\mu} \varphi_\mu f_j^\mu.$$
It is clear that $\{f_j\}\subset \mathcal{L}^{p,q}(U,H)$. If $K\subset\subset U$ is compact, then $K\cap \supp \varphi_\mu$
is a compact subset of $U_\mu$ for each $\mu$, so that $\{f_j^\mu\}$ and $\{\dq f_j^\mu\}$ converge in the $L^2$-sense
to $f_\mu$ resp. $\dq_w f_\mu$ on $K\cap\supp \varphi_\mu$. But then $\{f_j\}$ and $\{\dq f_j\}$ converge in the $L^2$-sense
to $f$ resp. $\dq_w f$ on $K$ (recall that the cover is locally finite) and that is what we had to show.

As for $\mathcal{C}^{p,q}(H)$, it is clear that the sheaves $\mathcal{F}^{p,q}(H)$ are fine,
and we obtain fine sequences
\begin{eqnarray}\label{eq:Fseq1}
\mathcal{F}^{p,0}(H) \overset{\dq_s}{\longrightarrow} \mathcal{F}^{p,1}(H) 
\overset{\dq_s}{\longrightarrow} \mathcal{F}^{p,2}(H) \overset{\dq_s}{\longrightarrow} ...
\end{eqnarray}
We can now introduce the sheaf
\begin{eqnarray}\label{eq:KXs}
\mathcal{K}^s_X(H) := \ker \dq_s \subset \mathcal{F}^{n,0}(H)
\end{eqnarray}
which we may call the {\it canonical sheaf of holomorphic $H$-valued $n$-forms with Dirichlet boundary condition}.
The main objective of this section is to compare different representations of the cohomology of $\mathcal{K}^s_X$.
One of them will be the $L^{2,loc}$-Dolbeault cohomology with respect to the $\dq_s$-operator on open sets $U\subset X$,
i.e. the cohomology of the complex \eqref{eq:Fseq1} which is denoted by $H^q(\Gamma(U,\mathcal{F}^{p,*}))$.

\subsection{Proof of Theorem \ref{thm:exactness2} -- $L^2$-resolution for $\mathcal{K}_X^s$}

We assume from now on that $X$ has only isolated singularities and that $L$ is a Hermitian
holomorphic line bundle over $X$, not just over $X^*$.

Then the complex
\begin{eqnarray}\label{eq:complexsL}
0\rightarrow \mathcal{K}_X^s(L) \hookrightarrow \mathcal{F}^{n,0}(L) \overset{\dq_s}{\longrightarrow}
\mathcal{F}^{n,1}(L) \overset{\dq_s}{\longrightarrow} \mathcal{F}^{n,2}(L) \overset{\dq_s}{\longrightarrow} ... \longrightarrow \mathcal{F}^{n,n}(L)
\rightarrow 0
\end{eqnarray}
is exact by use of Lemma \ref{lem:dolb3} which says that $H^{n,\mu}_{s,w}(B^*)=0$ for all $\mu\geq 1$
if $B$ is a very small ball centered at an isolated singularity.
It is clear that this statement does not depend on the line bundle $L\rightarrow X$.
We could prove exactness of \eqref{eq:complexsL} also for a Hermitian line bundle over $X^*$
which is locally semi-positive with respect to $X$, but that would require a full review of 
the methods of \cite{FOV2}.

It is clear that the sheaves $\mathcal{F}^{p,q}(L)$ are fine,
and so the proof of Theorem \ref{thm:exactness2} is complete.

\subsection{Proof of Theorem \ref{thm:Ks} -- Desingularization of $\mathcal{K}_X^s$}

Let $(X,h)$ be a Hermitian complex space with only isolated singularities.

Let us outline the proof of Theorem \ref{thm:Ks}.
We will first use results from \cite{OvVa4} to show that $\mathcal{K}_X^s$ is a torsion-free coherent analytic sheaf.
We can then use a Theorem of Rossi  \cite{Ro} and Hironaka's desingularization to obtain a resolution of singularities
(with only normal crossings)
$$\pi: M \rightarrow X$$
such that $\pi^* \mathcal{K}_X^s / \mathcal{T}(\pi^* \mathcal{K}_X^s)$ is a locally free subsheaf of $\mathcal{K}_M$.
Here, $\mathcal{T}(\pi^* \mathcal{K}_X^s)$ denotes the torsion sheaf of $\pi^* \mathcal{K}_X^s$.

It follows from Theorem \ref{thm:ac3} and Lemma \ref{lem:ac4} that there exists an effective divisor $D$ with only normal crossings
(with support on the exceptional set of the resolution) such that
\begin{eqnarray}\label{eq:gsp001}
\mathcal{K}_X^s = \pi_* \big( \mathcal{K}_M\otimes \OO_M(-D) \big).
\end{eqnarray}
We will then show that $D \geq Z-|Z|$ where 
$Z=\pi^{-1}(\Sing X)$ is the unreduced exceptional divisor and $|Z|$ the underlying reduced divisor.

\bigskip
If $n=\dim X=2$ or -- more general -- if the exceptional set of the resolution has only double self-intersections,
we can prove that \eqref{eq:gsp001} holds with $D=Z-|Z|$.

\newpage

\subsubsection{$\mathcal{K}_X^s$ is a coherent analytic sheaf}\label{sssec:fix1}

It is clear that $\mathcal{K}_X^s$ is an analytic subsheaf of the Grauert-Riemenschneider canonical sheaf $\mathcal{K}_X$.
Thus $\mathcal{K}_X^s$ is of relation finite type, but it is a bit more involved to see that it is finitely generated.

Consider the short exact sequence
\begin{eqnarray*}\label{eq:fix01}
0\rightarrow \mathcal{K}_X^s \longrightarrow \mathcal{K}_X \longrightarrow \mathcal{K}_X / \mathcal{K}_X^s \rightarrow 0.
\end{eqnarray*}
We will show that $\mathcal{K}_X/\mathcal{K}_X^s$ is coherent. Then it follows that $\mathcal{K}_X^s$ is coherent, too,
as it is well-known that $\mathcal{K}_X$ is coherent.

But $\mathcal{K}_X/\mathcal{K}_X^s$ is a skyscraper sheaf with support in the isolated singularities of $X$.
It is enough to consider an isolated singularity $x_0\in\Sing X$. We can assume that $X$ is embedded locally at $x_0$
into a complex number space $\C^N$ such that $x_0=0\in\C^N$. For $r>0$, let $B_r=\{z\in\C^N: \|z\|<r\}$,
$X_r=X\cap B_r$ and $X_r'=X\cap B_r - \{0\}$. Fix $r>0$ so small that $X_r$ contains no singularity besides the origin
and $X_r$ has a strictly pseudoconvex boundary in $\Reg X$.

We will show that
$$\big( \mathcal{K}_X / \mathcal{K}_X^s\big)_0 = (\mathcal{K}_{X})_0 / (\mathcal{K}_{X}^s)_0$$
is a $\C$-vector space of finite dimension (and then $\mathcal{K}_X/\mathcal{K}_X^s$ is clearly coherent).
For this, we define a linear map $\Psi: (\mathcal{K}_X)_0 \rightarrow \C^L$ such that $\ker \Psi=(\mathcal{K}_X^s)_0$.

It is known that $H^{0,n-1}_{max}(X_r')$ is of finite dimension (see \cite{OvVa4}, Section 2.3). 
We consider instead the finite-dimensional quotient space
$$ V:= \frac{H^{0,n-1}_{max}(X_r')}{\phi_*(H^{0,n-1}_{max}(\widetilde{X}_r))}$$
from the remark after Proposition 4.5 in \cite{OvVa4}. Let $f_1, ..., f_L$ be a basis of $V$.
For a germ $\omega\in (\mathcal{K}_X)_0$, we define for $j=1, ..., L$ the map
\begin{eqnarray}\label{eq:fix02}
\Psi_j (\omega) := \int_{X_r'} f_j \wedge \dq \chi\wedge \omega,
\end{eqnarray}
where the cut-off function $\chi$ is chosen as in \cite{OvVa4}, (12), i.e. $\chi$ is a smooth non-negative cut-off function with compact
support in $X_r$ that is identically $1$ in a neighborhood of the origin.
The support of $\chi$ has to be taken so small
that the integral makes sense (according to the domain where $\omega$ is defined).
This definition makes sense as the integral \eqref{eq:fix02} does not depend on the particular choice of
$\chi$ (as one sees easily by partial integration).

That defines a $\C$-linear map
$$\Psi=(\Psi_1, ..., \Psi_L): (\mathcal{K}_X)_0 \rightarrow \C^L,$$
and $(\mathcal{K}_X^s)_0 = \ker \Psi$ by the remark to Proposition 4.5 in \cite{OvVa4}.

\subsubsection{The monoidal transformation of $\mathcal{K}_X^s$}\label{sssec:fix2}

It is now clear that $\mathcal{K}_X^s$ is a torsion-free coherent analytic sheaf, which is generically locally free of rank one
(i.e. on the regular part of $X$). We wish to make $\mathcal{K}_X^s$ locally free.

By a Theorem of Rossi \cite{Ro} (see \cite{Rie}, Theorem 2, for a suitable reference), there exists a proper modification
$$\sigma: X' \rightarrow X$$
such that 
$$\sigma^T \mathcal{K}_X^s := \sigma^* \mathcal{K}_X^s / \mathcal{T}(\sigma^* \mathcal{K}_X^s)$$ 
is locally free of rank one.
Here, $\mathcal{T}(\sigma^* \mathcal{K}_X^s)$ is the torsion sheaf of $\sigma^* \mathcal{K}_X^s$,
and the torsion-free analytic preimage $\sigma^T \mathcal{K}_X^s$ is called the monoidal transformation of $\mathcal{K}_X^s$.

Now let
$\rho: M \rightarrow X'$
be a resolution of singularities with only normal crossings.
Then $\rho^* \sigma^T \mathcal{K}_X^s = \rho^T \sigma^T \mathcal{K}_X^s$ is again locally free of rank one.
Let from now on
$$\pi: M \rightarrow X$$
be the resolution of singularities $\sigma\circ\rho$ which has an exceptional divisor with normal crossings only.
Then
\begin{eqnarray*}
\pi^T \mathcal{K}_X^s := \pi^* \mathcal{K}_X^s / \mathcal{T}(\pi^* \mathcal{K}_X^s)
\end{eqnarray*}
is locally free of rank one because $\pi^T = \rho^T \sigma^T = \rho^* \sigma^T$ by the Korollar in \cite{GrRie}, §1.3.
As there is a natural injection $\pi^T \mathcal{K}_X^s \hookrightarrow \mathcal{K}_M$,
it follows from Lemma \ref{lem:ac4}
that there exists an effective divisor, $D\geq 0$, with support on the exceptional set such that
$$\pi^T \mathcal{K}_X^s= \mathcal{K}_M \otimes\OO(-D).$$
Finally, Theorem \ref{thm:ac3} yields:
\begin{eqnarray}\label{eq:gsp002}
\mathcal{K}_X^s = \pi_* \pi^T \mathcal{K}_X^s = \pi_* \big( \mathcal{K}_M\otimes \OO_M(-D) \big).
\end{eqnarray}
In the following, we will study the divisor $D$.

\medskip
\subsubsection{The $\dq_{s,E}$-operator on the resolution}\label{sssec:fix3}

Let $\pi: M\rightarrow X$ be a resolution of singularities with only normal crossings as above.
As in Section \ref{sssec:wresolution}, let $\gamma=\pi^* h$ be the positive semi-definite pseudometric
and give $M$ a freely chosen positive definite metric $\sigma$.

On $M$, we denote by $\dq_{s,E}$ the $\dq$-operator acting on $\mathcal{L}^{p,q}_\gamma$-forms,
defined as the $\dq_s$-operator on $X$ above, but with the exceptional set $E$ in place of the singular set $\Sing X$
(so that $\pi_* \dq_{s,E}=\dq_s$). Let
$$\mathcal{F}^{p,q}_{\gamma,E} := \mathcal{L}^{p,q}_\gamma \cap \dq_{s,E}^{-1} \mathcal{L}^{p,q+1}_\gamma.$$
Then it follows from \eqref{eq:l2est5} that $(\mathcal{F}^{p,*},\dq_s)$ can be canonically
identified with the direct image complex $(\pi_* \mathcal{F}^{p,*}_{\gamma,E},\pi_* \dq_{s,E})$.
We will show that
\begin{eqnarray}\label{eq:Ks}
\mathcal{F}^{n,0}_{\gamma,E} \cap \ker \dq_{s,E} \subset \mathcal{K}_M \otimes \OO(|Z|-Z)
\end{eqnarray}
for $X$ with only isolated singularities (Lemma \ref{lem:inclusion}).
Recall that $Z=\pi^{-1}(\Sing X)$ is the unreduced exceptional divisor and $|Z|$ the underlying reduced divisor.
On the other hand, we will show that the inverse inclusion holds in \eqref{eq:Ks} if
the resolution is chosen appropriately and
$n=\dim X=2$ or the exceptional set of the resolution has only double self-intersections (Lemma \ref{lem:fix1}).

\medskip
Let $U\subset M$ be an open set. We can assume
that $U$ is an open set in $\C^n$ and that the exceptional set $E$ is just the normal crossing $\{z_1\cdots z_d=0\}$.
We can assume further that $U\subset \pi^{-1}(U')$, where $U'$ is an open neighborhood of an isolated singularity
$p\in\Sing X$, and that $U'$ is embedded holomorphically in $\C^N$ such that $p=0\in\C^N$.
Let $w_1, ..., w_N$ be the Euclidean holomorphic coordinates of $\C^N$.
The unreduced exceptional divisor $Z=\pi^{-1}(\{0\})$ is given as the common zero set
of the holomorphic functions $\{\pi^* w_1, ..., \pi^*w_N\}$.
Let $Z$ have the order $k_j\geq 1$ on $\{z_j=0\}$, i.e. assume that $Z$ is given by
the holomorphic function $f=z_1^{k_1}\cdots z_d^{k_d}$. Let $k_{d+1}=...=k_n=0$ for ease of notation.
Let
$$F:= \left( \sum_{j=1}^N |w_j|^2\right)^{1/2}.$$
%Since the metric $h$ is quasi-isometric to the Euclidean metric in $\C^N$, we have that
%\begin{eqnarray}\label{eq:gsp01}
%|\pi^* \dq F|_\gamma = |\dq F|_h \lesssim 1.
%\end{eqnarray}
We need to understand $\pi^* F$. For $j=1, ..., N$, we have that
\begin{eqnarray*}%\label{eq:gsp2}
\pi^* w_j = z_1^{k_1(j)} \cdots z_n^{k_n(j)} \cdot H_j,
\end{eqnarray*}
where $H_j$ is a non-vanishing holomorphic function and $k_\mu(j)\geq k_\mu = \min_j\{k_\mu(j)\}$ 
for all $\mu=1, ..., n$. It follows that
\begin{eqnarray*}
(\pi^* F)^2 &=& \sum_{j=1}^N |\pi^* w_j|^2 \\%\label{eq:gsp1}
&=& |z_1^{k_1}\cdots z_d^{k_d}|^2 
\left( \sum_{j=1}^N |z_1^{k_1(j)-k_1}\cdots z_n^{k_n(j)-k_n}|^2 \right) \wt{H},
\end{eqnarray*}
where $\wt{H}$ is a positive smooth function. 
At this place we have to assume that the resolution of singularities $\pi: M\rightarrow X$ is chosen appropriately.
By use of Hironaka's resolution of singularities we can assume that $(\pi^* F)^2$ is already monomial in the sense that
\begin{eqnarray}\label{eq:gsp003}
(\pi^* F)^2 = \sum_{j=1}^N |\pi^* w_j|^2  = |z_1^{k_1}\cdots z_d^{k_d}|^2 \widehat{H},
\end{eqnarray}
where $\widehat{H}$ is a positive smooth function. This is achieved by resolving also the sheaf of ideals
$\mathcal{I}=(\pi^* w_1, ..., \pi^* w_N)$, i.e. making the strict transform of $\mathcal{I}$ locally free.

\begin{lem}\label{lem:fix1}
Assume that $d\leq 2$ in the situation above. Then:
\begin{eqnarray}\label{eq:fix05}
\left.\mathcal{K}_M\otimes \OO(|Z|-Z)\right|_U \subset \left.\mathcal{F}^{n,0}_{\gamma,E}\right|_U.
\end{eqnarray}
\end{lem}

\begin{proof}
For the proof, we use a standard cut-off procedure. Let $U$ be as above and let $\phi\in \Gamma(U,\mathcal{K}_M\otimes \OO(|Z|-Z))$. 
Then $\phi\in \mathcal{L}^{n,0}_\gamma(U)$, and we have to show that $\phi\in\Dom\dq_{s,E}$.
The statement is local, so it is enough to prove that $\chi \phi \in\Dom\dq_{min}(U\setminus E)$ in $L^{n,0}_\gamma(U\setminus E)$,
where $\chi\in C^\infty_{cpt}(U)$ is a smooth cut-off function which is identically $1$ on an arbitrary large subset of $U$.\footnote{
Recall Section \ref{ssec:l2serre} for the $\dq_{min}$-operator which we consider on the Hermitian manifold $(U\setminus E,\gamma)$.}
Let $\h{\phi}:=\chi\phi$ and $K\subset U$ the support of $\chi$ which is compact in $U$ (but not in $U\setminus E$).

As in \cite{PS1}, Lemma 3.6, let $\rho_k: \R\rightarrow [0,1]$, $k\geq 1$, be smooth cut-off functions
satisfying 
$$\rho_k(x)=\left\{\begin{array}{ll}
1 &,\  x\leq k,\\
0 &,\  x\geq k+1,
\end{array}\right.$$
and $|\rho_k'|\leq 2$. Moreover, let $r: \R\rightarrow [0,1/2]$ be a smooth increasing function such that
$$r(x)=\left\{\begin{array}{ll}
x &,\ x\leq 1/4,\\
1/2 &,\ x\geq 3/4,
\end{array}\right.$$
and $|r'|\leq 1$.
We do need a function measuring the distance to the exceptional set $E$ in $M$.
A good choice is just the pull-back of the Euclidean distance in $\C^N$,
i.e. the function $\pi^*F$ from above.
Since the metric $h$ is quasi-isometric to the Euclidean metric in $\C^N$, we have $|\dq F|_h\lesssim 1$.
As cut-off functions we can use 
$$\mu_k:=\rho_k(\log(-\log r(\pi^* F)))$$ 
on $M$. Thus, we claim that
$$\phi_k:= \mu_k\h{\phi}$$
is a suitable sequence of smooth forms with support away from $E$.
%Let $K\subset U$ be the support of the cut-off function $\chi$ in $U$.
It is clear that $\phi_k\rightarrow\h{\phi}$ in $L^{n,0}_\gamma(U\setminus E)$
and that $\mu_k \dq\h{\phi} \rightarrow \dq\h{\phi}$ in $L^{n,1}_\gamma(U\setminus E)$ as $k\rightarrow\infty$.
We will now show that
\begin{eqnarray}\label{eq:co1}
\dq\mu_k \wedge\h{\phi} \rightarrow 0
\end{eqnarray}
weakly in $L^{n,1}_\gamma(U\setminus E)$ as $k\rightarrow\infty$.

By definition,
\begin{eqnarray}
|\dq \mu_k|_\gamma^2 &\leq& \frac{|\rho_k'(\log(-\log(r(\pi^*F))))|^2}{r^2(\pi^*F)\log^2(r(\pi^*F))} |r'|^2 |\dq\pi^*F|_\gamma^2\\
&\lesssim& \frac{\chi_k(\pi^*F)}{(\pi^*F)^2\log^2(\pi^*F)},\label{eq:co2}
\end{eqnarray}
where $\chi_k$ is the characteristic function of $[e^{-e^{k+1}},e^{-e^k}]$ as $|\pi^*\dq F|_\gamma=|\dq F|_h\lesssim 1$ and 
$\mu_k$ is constant outside $[e^{-e^{k+1}},e^{-e^k}]$.

\bigskip
By the assumption $d\leq 2$, it follows from \eqref{eq:gsp003} that
$\pi^*F \sim |z_1^{k_1} z_2^{k_2}|$, and \eqref{eq:co2} yields
\begin{eqnarray}\label{eq:co3}
|\dq\mu_k|_\gamma \lesssim \chi_k(\pi^* F) |z_1|^{-k_1} |z_2|^{-k_2} |\log|^{-1}\big( |z_1 z_2|\big).
\end{eqnarray}
The assumption $\phi\in \Gamma(U,\mathcal{K}_M\otimes \OO(|Z|-Z))$ implies that
$$z_1^{1-k_1} z_2^{1-k_2} \phi \in \Gamma(U,\mathcal{K}_M)$$
which does almost compensate the right hand-side of \eqref{eq:co3}. 
We have to take care of the additional factor $\lambda:=|z_1 z_2||\log||z_1 z_2|$.
But $z_1^{1-k_1} z_2^{1-k_2} \phi$ has smooth coefficients (i.e. bounded on $K=\supp\chi$),
and for an $(n,0)$-form the $\gamma$-norm equals the $\sigma$-norm.
So, $|\dq\mu_k \wedge \h{\phi}|_\gamma = |\dq\mu_k|_\gamma|\h{\phi}|_\gamma \lesssim \chi_k(\pi^* F) \lambda^{-1}$,
and it is enough to show that 
\begin{eqnarray}\label{eq:gsp004}
\int_K \frac{\chi_k(|z_1^{k_1} z_2^{k_2}|)}{|z_1 z_2|^2 |\log|^2 |z_1 z_2|}
\end{eqnarray}
is uniformly bounded in $k$.
Combining this with \eqref{eq:co3}, we see that $\dq\mu_k\wedge\h{\phi}$ is uniformly bounded in $L^{n,1}_\gamma(U\setminus E)$,
and so $\dq\mu_k\wedge\h{\phi} \rightarrow 0$ weakly in $L^{n,1}_\gamma(U\setminus E)$
as desired because the domain of integration vanishes as $k\rightarrow\infty$.

The fact that \eqref{eq:gsp004} is uniformly bounded is shown in an appendix, Section \ref{sec:appB}.

\medskip
We can now conclude that actually $\h{\phi}\in\Dom\dq_{min}(U\setminus E)$.
To see this, recall that $\dq_{min}^* = \theta_{max}$ in the $L^2_\gamma$-sense on $U\setminus E$ (see \eqref{eq:adjoint2}).
But the considerations above yield
\begin{eqnarray*}
(\h{\phi},\theta_{max} g)_{U\setminus E} &=& \lim_{k\rightarrow \infty} (\phi_k, \theta_{max} g)_{U\setminus E}
= \lim_{k\rightarrow\infty} (\dq\phi_k, g)_{U\setminus E}\\
&=& (\dq \h{\phi},g)_{U\setminus E} + \lim_{k\rightarrow \infty} (\dq\mu_k\wedge \h{\phi}, g)_{U\setminus E}
= (\dq \h{\phi},g)_{U\setminus E}
\end{eqnarray*}
for all $g\in\Dom\theta_{max} \subset L^{n,1}_\gamma(U\setminus E)$.
For the partial integration, we have used the fact that the $\phi_k$ have compact support in $U\setminus E$.
So, $\h{\phi}\in\Dom\dq_{min}(U\setminus E)$ as desired.
\end{proof}

Let us now prove \eqref{eq:Ks}. This inclusion holds for arbitrary exceptional divisor (with only normal crossings).
We show a bit more which will be of use later in the proof of Theorem \ref{thm:B}.
Recall that $L_{|Z|-Z}\rightarrow M$ is a Hermitian holomorphic line bundle such that holomorphic sections of $L_{|Z|-Z}$
correspond to sections in $\OO(|Z|-Z)$. Hence
$$\mathcal{C}^{n,0}_\sigma(L_{|Z|-Z})\cap \ker \dq_w = \mathcal{K}_M\otimes\OO(|Z|-Z),$$
and so \eqref{eq:Ks} follows from the following:

\begin{lem}\label{lem:inclusion}
For all $p\geq 0$, we have
$$\mathcal{F}^{n,p}_{\gamma,E} \subset \mathcal{C}^{n,p}_\sigma(L_{|Z|-Z})$$
as subsheaves of $\mathcal{L}^{n,p}_\sigma$. It follows that
$$\mathcal{F}^{n,0}_{\gamma,E}\cap \ker \dq_{s,E} \subset \mathcal{K}_M\otimes \OO(|Z|-Z),$$
and
$$\mathcal{F}^{n,p} \cong \pi_* \mathcal{F}^{n,p}_{\gamma,E} \subset \pi_* \mathcal{C}^{n,p}_\sigma(L_{|Z|-Z}).$$
\end{lem}

\begin{proof}
We observe that $\mathcal{F}^{n,p}_{\gamma,E} \subset \mathcal{L}^{n,p}_\gamma \subset \mathcal{L}^{n,p}_\sigma$
by definition of $\mathcal{F}^{n,p}_{\gamma,E}$ and \eqref{eq:l2est3}.
On the other hand, there is a natural inclusion
$\mathcal{L}^{n,p}_\sigma(L_{|Z|-Z}) \subset \mathcal{L}^{n,p}_\sigma$
since $Z-|Z|$ is an effective divisor so that
$\mathcal{C}^{n,p}_\sigma(L_{|Z|-Z}) \subset \mathcal{L}^{n,p}_\sigma(L_{|Z|-Z}) \subset \mathcal{L}^{n,p}_\sigma$
by definition of $\mathcal{C}^{n,p}_\sigma(L_{|Z|-Z})$.

As the statement is local, it is enough to consider a point $P\in E$ and a neighborhood $U$ of $P$ such that
$U$ is an open set in $\C^n$, that $E$ is the normal crossing $\{z_1\cdots z_d=0\}$,
and $P=0$. For $\sigma$ we can take the Euclidean metric.

Let us investigate the behavior of $(0,1)$-forms under the resolution $\pi: M\rightarrow X$
at the isolated singularity $\pi(P)$. We can assume that a neighborhood of $\pi(P)$ is embedded
holomorphically into $W\subset\subset\C^L$, $L\gg n$, such that $\pi(P)=0$,
and that $\gamma=\pi^* h$ where $h$ is the Euclidean metric in $\C^L$.
Let $w_1, ..., w_L$ be the Cartesian coordinates of $\C^L$.
We are interested in the behavior of the forms $\eta_\mu:=\pi^* d\o{w_\mu}$ at the exceptional set.
Let $dz_N:=dz_1\wedge\cdots \wedge dz_n$. It follows from the observations in Section \ref{sssec:wresolution}
that a form $\alpha$ is in $L^{n,q}_{\gamma}(U)$ exactly if it can be written
in multi-index notation as
\begin{eqnarray}\label{eq:alpha01}
\alpha = \sum_{|K|=q} \alpha_K dz_N\wedge \eta_K = dz_N \wedge \sum_{|K|=q} \alpha_K \eta_K
\end{eqnarray}
with coefficients $\alpha_K \in L^{0,0}_\sigma(U)$. 
That can be seen as follows. Since the forms $\eta_K$ are orthogonal to $dz_N$,
we have
$$|\alpha|_\gamma = |dz_N|_\gamma \big|\sum_{|K|=q} \alpha_K \eta_K\big|_\gamma.$$
Let $g$ be a function as in Section \ref{sssec:wresolution}, i.e. $dV_\gamma=g^2dV_\sigma$.
As $|dz_N|_\gamma = |g|^{-1}$,
there are coefficients $\alpha_K$ in \eqref{eq:alpha01} such that
$$|\alpha|_\gamma^2 = |g|^{-2} \sum_{|K|=q} |\alpha_K|^2.$$
So, $\alpha$ is in $L^{n,q}_{\gamma}(U)$ exactly if $|\alpha|_\gamma$ is in $L^{0,0}_\gamma(U)$
which is the case exactly if all the $g^{-1} \alpha_K$ are in $L^{0,0}_\gamma(U)$.
By use of $dV_\gamma=g^2dV_\sigma$, this is the case exactly if all the $\alpha_K$ are in $L^{0,0}_\sigma(U)$.
The representation \eqref{eq:alpha01} is not unique.

Let $Z$ have the order $k_j\geq 1$ on $\{z_j=0\}$, i.e. assume that $Z$ is given by
$f=z_1^{k_1}\cdots z_d^{k_d}$.
Since $Z=\pi^{-1}(\Sing X)$,
each $\pi^* w_\mu$ must vanish of order $k_j$ on $\{z_j=0\}$.
We conclude that $\pi^* w_\mu$
has a factorization
\begin{eqnarray*}
\pi^* w_\mu = f g_\mu = z_1^{k_1} \cdots z_d^{k_d}\cdot g_\mu,
\end{eqnarray*}
where $g_\mu$ is a holomorphic function on $U$. So,
\begin{eqnarray*}%\label{eq:psi4}
\eta_\mu = \pi^* d\o{w_\mu} = d\pi^*\o{w_\mu} = \big(\o{z_1}^{k_1-1}\cdots \o{z_d}^{k_d-1}\big) \cdot \beta_\mu,
\end{eqnarray*}
where the $\beta_\mu$ are $(0,1)$-forms that are bounded with respect to the non-singular metric $\sigma$.
This means that $\eta_\mu = \pi^*d\o{w_\mu}$ vanishes at least to the order of $Z-|Z|$ along the exceptional set $E$.

So, \eqref{eq:alpha01} implies that a form $\alpha$ is in $L^{n,q}_{\gamma}(U)$ exactly if it can be written
in multi-index notation as
\begin{eqnarray}\label{eq:alpha02}
\alpha = \big(\o{z_1}^{k_1-1}\cdots \o{z_d}^{k_d-1}\big)^q \sum_{|K|=q} \alpha_K dz_N\wedge \beta_K
\end{eqnarray}
with coefficients $\alpha_K \in L^{0,0}_\sigma(U)$. 

We conclude that
$\mathcal{F}^{n,p}_{\gamma,E} \subset \mathcal{L}^{n,p}_\gamma \subset \mathcal{L}^{n,p}_\sigma(L_{|Z|-Z})$
for all $p\geq 1$, and it remains to treat the case $p=0$.
So, let $\phi \in \mathcal{F}^{n,0}_{\gamma,E}(U)$.
This means that there exists $\psi\in \mathcal{L}^{n,1}_\gamma(U) \subset \mathcal{L}^{n,1}_\sigma(L_{|Z|-Z})(U)$
such that $\dq_{s,E} \phi=\psi$.
But this implies that $\phi \in \mathcal{L}^{n,0}_\sigma(L_{|Z|-Z})$ as we will show now.

The key point is that there exists a sequence of smooth forms $\phi_j$ with support away from $E$
such that 
\begin{eqnarray*}
\phi_j \rightarrow \phi &\mbox{ in }& L^{n,0}_\gamma(V) = L^{n,0}_\sigma(V),\\
\dq \phi_j \rightarrow \psi &\mbox{ in }& L^{n,1}_\gamma(V) \subset L^{n,1}_\sigma(V,L_{|Z|-Z})
\end{eqnarray*}
on suitable open sets $V\subset U$. The considerations above show that convergence in $L^{n,1}_\gamma(V)$
implies convergence in $L^{n,1}_\sigma(V,L_{|Z|-Z})$.

As we treat a local question at $0\in\C^n$, it does no harm to work on a suitable neighborhood of the origin and
to cut-off $\phi$ and the $\phi_j$
by a real-valued smooth function $\chi\in C^\infty_{cpt}(\C)$ with $\chi(z_1)=1$ for $|z_1|\leq \epsilon$,
$\chi(z_1)=0$ for $|z_1|\geq 2\epsilon$, and $|\chi'|\leq 2\epsilon^{-1}$ for a fixed $\epsilon>0$ small enough.
So, replace $\phi(z)$ by $\phi(z)\chi(z_1)$ and $\phi_j(z)$ by $\phi_j(z)\chi(z_1)$.
The new $(n,0)$-forms are not holomorphic any more.

As the $\phi_j$ have compact support away from $E$, we have the representation
\begin{eqnarray}\label{eq:cif}
\phi_j(z) = \frac{z_1^{k_1-1}}{2\pi i} \int_\C \frac{\partial \phi_j}{\partial\o{\zeta_1}}(\zeta_1, z_2, ..., z_n) 
\frac{d\zeta_1\wedge d\o{\zeta_1}}{\zeta_1^{k_1-1}(\zeta_1-z_1)},
\end{eqnarray}
omitting $dz_N$ in the notation for simplicity. But $\dq\phi_j\rightarrow \psi=\dq_w\phi$
in $L^{n,1}_\gamma(V)$ and the representation \eqref{eq:alpha02}
imply that 
$$\zeta_1^{-k_1+1} \dq \phi_j \rightarrow \zeta_1^{-k_1+1} \dq \phi$$
in the $L^2$-sense with respect to the non-singular metric $\sigma$.
But the Cauchy formula \eqref{eq:cif} is bounded as an operator $L^2\rightarrow L^2$.
Hence, the formula \eqref{eq:cif} converges to
$$\phi(z) = \frac{z_1^{k_1-1}}{2\pi i} \int_\C \frac{\partial\phi}{\partial\o{\zeta_1}}(\zeta_1, z_2, ..., z_n) 
\frac{d\zeta_1\wedge d\o{\zeta_1}}{\zeta_1^{k_1-1}(\zeta_1-z_1)},$$
and the integral on the right-hand side is in $L^{n,0}_\sigma$.
Thus, we obtain $z_1^{1-k_1} \phi\in L^{n,0}_\sigma$. Similarly, we have $z_j^{1-k_j}\phi \in L^{n,0}_\sigma$
for $j=2, ..., d$. But $\phi$ is an ordinary (smooth) holomorphic $(n,0)$-form.
It follows that $\phi\in \mathcal{L}^{n,0}_{\sigma}(U,L_{|Z|-Z})$.

So, we have seen that
\begin{eqnarray}\label{eq:inc1}
\mathcal{F}^{n,p}_{\gamma,E} &\subset& \mathcal{L}^{n,p}_\sigma(L_{|Z|-Z})\ ,\ p\geq 0,\\
\mathcal{L}^{n,p}_\gamma &\subset& \mathcal{L}^{n,p}_\sigma(L_{|Z|-Z})\ ,\ p\geq 1.\label{eq:inc2}
\end{eqnarray}
Let $p\geq 0$ and $\phi\in \mathcal{F}^{n,p}_{\gamma,E}(U)$ on an open set $U$.
Then $\phi\in \mathcal{L}^{n,p}_\sigma(U,L_{|Z|-Z})$ by \eqref{eq:inc1} 
and $\dq_{s,E}\phi\in \mathcal{L}^{n,p+1}_\sigma(U,L_{|Z|-Z})$ by \eqref{eq:inc2}.
It follows that $\dq_w\phi=\dq_{s,E}\phi$ in $\mathcal{L}^{n,p+1}_\sigma(U,L_{|Z|-Z})$ and thus $\phi\in \mathcal{C}^{n,p}_\sigma(L_{|Z|-Z})(U)$.
\end{proof}

\subsubsection{Properties of the divisor $D$ in $\mathcal{K}_X^s=\pi_*\big( \mathcal{K}_X\otimes\OO(-D)\big)$}

As the direct image functor is left-exact, it follows from \eqref{eq:Ks} (or Lemma \ref{lem:inclusion}, respectively)
that
$$\mathcal{K}_X^s = \ker \dq_s = \pi_* (\ker\dq_{s,E}) \subset \pi_* \big( \mathcal{K}_M \otimes\OO(|Z|-Z)\big).$$
But then
\begin{eqnarray*}
\mathcal{K}_M \otimes \OO(-D) = \pi^T \mathcal{K}_X^s \subset \pi^T \pi_* \big( \mathcal{K}_M \otimes\OO(|Z|-Z)\big) 
\subset \mathcal{K}_M \otimes\OO(|Z|-Z).
\end{eqnarray*}
Here, the first inclusion is valid as we consider the torsion-free analytic preimages
%(the holomorphic $n$-forms are uniquely determined by their values on the regular part of the variety or on $M-E$, respectively)
(the functor $\pi^T$ is exact on coherent analytic sheaves, see Section \ref{sec:appC}),
and the second inclusion comes from Lemma \ref{lem:ac1}. Thus: $D\geq Z-|Z|$.

\medskip
It remains to consider the situation when the exceptional set $E$ has only double self-intersections, i.e.
$d\leq 2$ in Section \ref{sssec:fix3}. This covers particularly the case of dimension $n=\dim X=2$.
But then
\begin{eqnarray*}%\label{ex:fix005}
\mathcal{K}_X^s = \ker \dq_s = \pi_* (\ker\dq_{s,E}) = \pi_* \big( \mathcal{K}_M \otimes\OO(|Z|-Z)\big)
\end{eqnarray*}
by Lemma \ref{lem:fix1} and Lemma \ref{lem:inclusion}.
Here, Sections \ref{sssec:fix1} and \ref{sssec:fix2} are not necessary
and $\mathcal{K}_X^s$ is coherent by Grauert's direct image theorem.

\subsection{Proof of Theorem \ref{thm:B} and Theorem \ref{thm:Bd}}

Let $\pi: M \rightarrow X$ be a resolution of singularities as in Theorem \ref{thm:Ks}.
We see by Theorem \ref{thm:exactness2} that $(\mathcal{F}^{n,*},\dq_s)$ is a fine resolution for $\mathcal{K}_X^s$,
and by Theorem \ref{thm:Ks} that $\mathcal{K}_X^s=\pi_* \big(\mathcal{K}_M\otimes\OO(-D)\big)$
with an appropriate effective divisor $D\geq Z-|Z|$.

We can prove Theorem \ref{thm:B} now by use of the Leray spectral sequence.
For ease of notation, we write $\mathcal{S}:=\mathcal{K}_M\otimes\OO(-D)$.
Then the Leray spectral sequence for $\mathcal{S}$ and $\pi: M\rightarrow X$ is the spectral sequence
given by
\begin{eqnarray*}
E^{p,q}_2:= H^p(X,R^q\pi_* \mathcal{S})
\end{eqnarray*}
and it converges to $H^k(M,\mathcal{S}) = \oplus_{p+q=k} E^{p,q}_\infty$. But $X$ has only isolated singularities
so that the $R^q\pi_* \mathcal{S}$, $q>0$, are skyscraper sheaves with support in the singular set of $X$.
Thus $E^{p,q}_2=H^p(X,R^q\pi_*\mathcal{S})=0$ if $p>0$ and $q>0$.

For such a spectral sequence, there is for each $k\geq 1$ a natural short exact sequence
$0 \rightarrow E^{k,0}_2 \rightarrow H^k(K^*) \rightarrow E^{0,k}_2 \rightarrow 0$
where the injection and surjection are edge homomorphisms (see Appendix A, Lemma \ref{lem:appendix2}).
In our situation that means nothing else but exactness of
\begin{eqnarray*}
0 \rightarrow H^k(X,\pi_* \mathcal{S}) \rightarrow H^k(M,\mathcal{S}) \rightarrow \Gamma(X,R^k\pi_* \mathcal{S})\rightarrow 0
\end{eqnarray*}
As $(\mathcal{F}^{n,*},\dq_s)$ is a fine resolution for
$\pi_* \mathcal{S}=\mathcal{K}_X^s$ and $(\mathcal{C}^{n,*}_\sigma(L_{-D}),\dq_w)$ is a fine resolution for $\mathcal{S}$,
we get the exact sequence
\begin{eqnarray*}%\label{eq:final}
0 \rightarrow 
H^k\big(\Gamma(X,\mathcal{F}^{n,*})\big) \overset{i}{\longrightarrow} H^k\big(\Gamma(M,\mathcal{C}^{n,*}_\sigma(L_{-D}))\big)
\rightarrow \Gamma(X,R^k\pi_* \mathcal{S}) \rightarrow 0,
\end{eqnarray*}
where the projection is given is follows: any $\dq_w$-closed form $\phi\in \Gamma(M,C^{n,k}_\sigma(L_{-D}))$
defines naturally a global section in $\Gamma\big(X,R^k\pi_* (\mathcal{K}_M\otimes\OO(-D))\big)$.
Moreover, note that
$H^k\big(\Gamma(M,\mathcal{C}^{n,*}_\sigma(L_{-D}))\big)=H^k(M,\mathcal{K}_M\otimes\OO(-D))$
and $H^k\big(\Gamma(X,\mathcal{F}^{n,*})\big)=H^{n,k}_{min}(X^*)$. 

\smallskip
When $D=Z-|Z|$, we can explain the injection $i$ explicitely.
Let $\sigma$ be a positive definite Hermitian metric on $M$ and
$L_{|Z|-Z}\rightarrow M$ a Hermitian holomorphic line bundle such that holomorphic sections of $L_{|Z|-Z}$
correspond to sections in $\OO(|Z|-Z)$.

By use of Lemma \ref{lem:inclusion} we obtain an inclusion of complexes
\begin{eqnarray*}%\label{eq:morphB1}
\pi^*: (\mathcal{F}^{n,*},\dq_s) \rightarrow (\pi_* (\mathcal{C}^{n,*}_\sigma(L_{|Z|-Z})),\pi_* \dq_w)
\end{eqnarray*}
which induces (by $L^2$-extension of the $\dq_w$-equation over the execptional set) a morphism
on the cohomology of the complexes, representing $i$:
\begin{eqnarray*}%\label{eq:morphB2}
i=[\pi^*]: H^{n,q}_{min}(X^*)=H^q\big(\Gamma(X^*,\mathcal{F}^{n,*})\big) \longrightarrow H^q\big(\Gamma(M,\mathcal{C}^{n,*}_\sigma(L_{|Z|-Z}))\big).
\end{eqnarray*}

That proves Theorem \ref{thm:B}. 
Theorem \ref{thm:Bd} follows by $L^2$-duality,
Theorem \ref{thm:duality}, on $X$ and classical Serre duality on $M$ analogously to Theorem \ref{thm:Ad}.

\subsection{Proof of Theorem \ref{thm:ps2}}\label{ssec:ps2}

Let us first point out the difficulty in the proof of Theorem \ref{thm:ps2} in \cite{PS1}.
The critical point is the trace estimate (3.7) in the proof of \cite{PS1}, Lemma 3.6.
It is said that the trace estimate follows from the fact that $u^{-(m_1-1)}v^{-(m_2-1)}\psi$
is in the Sobolev-$1$-space. But consider the following example:
With $r(u,v)=\|(u,v)\|$ in $\C^2$, the function $f(u,v)=r^a$ is Sobolev-$1$ in a neighborhood of the origin when $a>-1$.
One can check that the integral of $f^2$ over $|uv|=t$ behaves as $t^{3/2+a}$ as $t\rightarrow 0^+$ when $a<-1/2$.
Hence $-1<a<-1/2$ gives counterexamples to the claim that Sobolev-$1$ implies the trace estimate.

So, the missing step in the proof of Pardon-Stern of Theorem \ref{thm:ps2} is to show that the natural map
$H^{1}(M,\OO(Z-|Z|)) \rightarrow H^{0,1}_{max}(X^*)$ is surjective (see \cite{PS1}, (3.1), Proposition 3.3 and Lemma 3.6).
But this statement is covered by our Theorem \ref{thm:Bd}
and so the proof of Theorem \ref{thm:ps2} is now complete.

\bigskip

%\newpage

\section{Appendix A -- Spectral Sequence of a Double Complex}

For the basic definitions, statements and notation we refer to \cite{De3}, IV.§11. {\it Spectral Sequence of a Double Complex}.
Let $K^{*,*} = \bigoplus K^{p,q}$ be a double complex with differential $d=d'+d''$ such that
$d': K^{p,q}\rightarrow K^{p+1,q}$, $d'': K^{p,q}\rightarrow K^{p,q+1}$.
Let $(K^*,d)$ be the simple complex associated to $K^{*,*}$, $K^l = \bigoplus_{l=p+q} K^{p,q}$.
We consider the spectral sequence associated to this double complex,
$$E^{p,q}_0=K^{p,q}\ ,\ E^{p,q}_1=H^q_{d''}(K^{p,*})\ ,\ E^{p,q}_2=H^p_{d'}(H^q_{d''}(K^{*,*}))$$
with differentials $d_0=d''$ and $d_1=d'$.
By definition of the spectral sequence,
there are for $\nu\geq 1$ natural exact sequences
\begin{eqnarray}\label{eq:4seq}
0 \rightarrow E^{0,\nu}_{\nu+2} \rightarrow E^{0,\nu}_{\nu+1} \overset{d_{\nu+1}}{\longrightarrow} 
E^{\nu+1,0}_{\nu+1} \rightarrow E^{\nu+1,0}_{\nu+2} \rightarrow 0.
\end{eqnarray}
In this appendix, we consider the special situation that
\begin{eqnarray}\label{eq:appendix1}
E^{p,q}_2 = 0\ \mbox{ if } p>0 \mbox{ and } q>0,
\end{eqnarray}
i.e. only the groups $E^{p,0}_2$, $E^{0,q}_2$ are non-zero.
Under this assumption, we get for $\nu\geq 1$ by definition of the spectral sequence:
\begin{eqnarray}\label{eq:deg1}
E^{0,\nu}_2= ...=E^{0,\nu}_\nu= &E^{0,\nu}_{\nu+1} &\supset E^{0,\nu}_{\nu+2} = ... =E^{0,\nu}_\infty,\\
E^{\nu,0}_2= ...=E^{\nu,0}_{\nu} \supset & E^{\nu,0}_{\nu+1} & = E^{\nu,0}_{\nu+2} = ... = E^{\nu,0}_\infty.\label{eq:deg2}
\end{eqnarray}

\begin{lem}\label{lem:appendix1}
If \eqref{eq:appendix1} holds, then there is a natural long exact sequence
\begin{eqnarray*}
0 &\rightarrow& E^{1,0}_2 \rightarrow H^1(K^*) \rightarrow E^{0,1}_2 
\overset{d_2}{\longrightarrow} E^{2,0}_2 \rightarrow H^2(K^*) \rightarrow E^{0,2}_3 \overset{d_3}{\longrightarrow} E^{3,0}_3 \rightarrow ...\\
... &\rightarrow& E^{\mu,0}_{\mu} \rightarrow H^\mu(K^*) \rightarrow E^{0,\mu}_{\mu+1} 
\overset{d_{\mu+1}}{\longrightarrow} E^{\mu+1,0}_{\mu+1} \rightarrow H^{\mu+1}(K^*) \rightarrow E^{0,\mu+1}_{\mu+2} \overset{d_{\mu+2}}{\longrightarrow} ...\ ,
\end{eqnarray*}
where the non indicated arrows are edge homomorphisms.
\end{lem}

\begin{proof}
\eqref{eq:appendix1} implies that $E^{p,q}_r=0$ for all $r\geq 2$ if $p>0$ and $q>0$.
So, $E^{p,q}_\infty=0$ if $p>0$ and $q>0$.
As the spectral sequence converges to the cohomology of $K^*$, $E^{p,q}_r \Rightarrow H^{p+q}(K^*)$,
we obtain for $\mu\geq 2$ short exact sequences
\begin{eqnarray}\label{eq:appendix2}
0 \rightarrow E^{\mu,0}_\infty \rightarrow H^\mu(K^*) \rightarrow E^{0,\mu}_\infty \rightarrow 0.
\end{eqnarray}
Merging these with the exact sequences \eqref{eq:4seq} by use of \eqref{eq:deg1} (i.e. $E^{0,\mu}_\infty=E^{0,\mu}_{\mu+2}$)
and \eqref{eq:deg2} (i.e. $E^{\mu,0}_\infty=E^{0,\mu}_{\mu+1}$) gives the result.
\end{proof}

\begin{lem}\label{lem:appendix2}
Assume that \eqref{eq:appendix1} holds and that the natural maps $H^l(K^*) \rightarrow E^{0,l}_2$ are surjective
for all $l\geq 1$. Then there is for all $k\geq 1$ a short exact sequence
\begin{eqnarray}\label{eq:appendix3}
0 \rightarrow E^{k,0}_2 \rightarrow H^k(K^*) \rightarrow E^{0,k}_2 \rightarrow 0
\end{eqnarray}
where the injection and surjection are edge homomorphisms.
\end{lem}

\begin{proof}
We use the sequence from Lemma \ref{lem:appendix1}. By \eqref{eq:deg1}, $E^{0,\mu}_{\mu+1}=E^{0,\mu}_2$,
so that the surjectivity assumption implies that the maps $d_{\mu+1}: E^{0,\mu}_{\mu+1}\rightarrow E^{\mu+1,0}_{\mu+1}$ vanish
by exactness of the long sequence from Lemma \ref{lem:appendix1}.
But this implies by use of \eqref{eq:deg1},\eqref{eq:deg2} and the definition of the
spectral sequence that $E^{0,\mu}_2 = E^{0,\mu}_\infty$, $E^{\mu,0}_2=E^{\mu,0}_\infty$.
The result follows by splitting the long exact sequence from Lemma \ref{lem:appendix1}.
\end{proof}

\smallskip

%\newpage

\section{Appendix B -- Computation of the integral in Lemma \ref{lem:fix1}}\label{sec:appB}

For a compact subset $K\subset \C^2$, we shall compute 
\begin{eqnarray*}%\label{eq:gsp007}
I_k := \int_K \frac{\chi_k(|z_1^{k_1} z_2^{k_2}|)}{|z_1 z_2|^2 |\log|^2 |z_1 z_2|} dV_{\C^2}(z_1,z_2)
\end{eqnarray*}
from the proof of Lemma \ref{lem:fix1}.
As the critical point is just the origin, we can assume that $K=\{z: |z_1|,|z_2| <e^{-1}\}$.
In polar coordinates, we have
\begin{eqnarray*}
I_k \sim \int_{\substack{r_1 < e^{-1},\\ r_2<e^{-1}}} \frac{\chi_k(r_1^{k_1} r_2^{k_2}) dr_1 dr_2}{ r_1 r_2 (- \log(r_1 r_2))^2}
= \int_{\substack{t_1>1,\\ t_2>1}} \frac{\chi_k(e^{-t_1k_1-t_2k_2}) dt_1 dt_2}{(t_1+t_2)^2},
\end{eqnarray*}
where the second step is the substitution $t_1=-\log(r_1)$, $t_2=-\log(r_2)$.
Another substitution $u=t_1+t_2$ yields
\begin{eqnarray*}
I_k \sim \int_\Delta \frac{\chi_k(e^{-t_1k_1-t_2k_2}) dt_1 du}{u^2},
\end{eqnarray*}
where the domain of integration is now $\Delta=\{(t_1,u)\in\R^2: u\geq 2 \mbox{ and } 1\leq t_1 \leq u-1\}$.
Recall that $\chi_k$ is the characteristic function of $[e^{-e^{k+1}},e^{-e^k}]$.
As $u=t_1+t_2$, there exist constants $c_0,c_1>0$ such that $u/c_1 \leq k_1t_1 + k_2 t_2 \leq u/c_0$.
So, if $u\notin [c_0e^k,c_1e^{k+1}]$, then $t_1k_1+t_2k_2\notin [e^k,e^{k+1}]$. Thus, we obtain further:
\begin{eqnarray*}
I_k &\lesssim& \int_{\substack{c_0 e^k \leq u\leq c_1 e^{k+1}\\ 1\leq t_1\leq u-1}} \frac{dt_1 du}{u^2} = \int_{c_0 e^k}^{c_1 e^{k+1}} \frac{(u-2)du}{u^2}
< \log\left(\frac{c_1 e^{k+1}}{c_0 e^k}\right) = \log\big(\frac{c_1}{c_0}\big)+1.
\end{eqnarray*}
So, the integral is in fact uniformly bounded, not depending on $k$.

%\smallskip

\section{Appendix C -- Modifications of canonical sheaves}\label{sec:appC}

In order to show that $\mathcal{K}_X^s=\pi_*\pi^T \mathcal{K}_X^s$, we need a few lemmata:

\begin{lem}\label{lem:ac1}
Let $\pi: Y \rightarrow X$ be a modification between locally irreducible complex spaces $Y$, $X$.
If $\mathcal{F}$ is a torsion-free coherent analytic sheaf on $X$,
then the canonical morphism 
\begin{eqnarray}\label{eq:ac1}
\mathcal{F} \rightarrow \pi_* \pi^T \mathcal{F}
\end{eqnarray}
is injective, where $\pi^T \mathcal{F}$ denotes the torsion-free analytic preimage.

If $\mathcal{G}$ is a coherent analytic sheaf on $Y$,
then the canonical morphism
\begin{eqnarray}\label{eq:ac2}
\pi^T \pi_* \mathcal{G} \rightarrow \mathcal{G}
\end{eqnarray}
is injective (here, $X$ need not be locally irreducible).
\end{lem}

\begin{proof}
$\pi$ is biholomorphic outside the exceptional set of the modification.
So, a germ in the kernel of \eqref{eq:ac1} must have support contained in the exceptional set.
Thus, it is itself already the zero-germ because $\mathcal{F}$ is torsion-free.

Analogously, \eqref{eq:ac2} is injective because $\pi^T \pi_* \mathcal{G}$ is torsion-free.
\end{proof}

By a similar argument, it is easy to see that the functor $\pi^T$ is exact if $\pi: Y\rightarrow X$ is a proper modification
and $Y$ is locally irreducible (we will use that fact below).

\begin{lem}\label{lem:ac2}
Let $\pi: Y\rightarrow X$ be a modification between reduced complex spaces $Y$, $X$ of pure dimension
such that $Y$ is locally irreducible,
and $\mathcal{K}_X$ the Grauert-Riemenschneider canonical sheaf on $X$.
Then the canonical morphism
\begin{eqnarray}\label{eq:ac3}
\mathcal{K}_X \rightarrow \pi_* \pi^T \mathcal{K}_X
\end{eqnarray}
is an isomorphism.
\end{lem}

\begin{proof}
First, \eqref{eq:ac3} is injective by the first statement of Lemma \ref{lem:ac1}
(by definition of $\mathcal{K}_X$, we do not require $X$ to be locally irreducible).
Let $\mathcal{K}_Y$ be the Grauert-Riemenschneider canonical sheaf of $Y$. 
Then the natural morphism $\pi^T \pi_* \mathcal{K}_Y \rightarrow \mathcal{K}_Y$ is injective
by the second statement of Lemma \ref{lem:ac1}.
But the push-forward $\pi_*$ is left-exact 
and $\pi_* \mathcal{K}_Y = \mathcal{K}_X$. So, it follows that there is also a natural injection
$\pi_* \pi^T \mathcal{K}_X \rightarrow \mathcal{K}_X$ (which is inverse to $\mathcal{K}_X \rightarrow \pi_* \pi^T \mathcal{K}_X$).
\end{proof}

\begin{thm}\label{thm:ac3}
Let $X$ be a Hermitian complex space of pure dimension $n$
and $\mathcal{K}_X^s$ the canonical sheaf of holomorphic $n$-forms with Dirichlet boundary condition defined above.
Let $\pi: M\rightarrow X$ a resolution of singularities. Then the canonical morphism
\begin{eqnarray}\label{eq:ac4}
\mathcal{K}_X^s \rightarrow \pi_* \pi^T \mathcal{K}_X^s
\end{eqnarray}
is an isomorphism.
\end{thm}

\begin{proof}
First, \eqref{eq:ac4} is injective by the first statement of Lemma \ref{lem:ac1} (and by the definition of $\mathcal{K}_X^s$,
$X$ need not be locally irreducible).
Second, consider the sheaf $\mathcal{G} := \mathcal{F}^{n,0}_{\gamma,E} \cap \ker \dq_{s,E}$ which is a subsheaf of $\mathcal{K}_M$
such that $\pi_* \mathcal{G} = \mathcal{K}_X^s$ (see Section \ref{sssec:fix3}).
The natural injection $\mathcal{G} \rightarrow \mathcal{K}_M$ induces the commutative diagram
\begin{eqnarray*}%\label{eq:diagram1}
\begin{xy}
  \xymatrix{
     \pi^T\pi_* \mathcal{G} \ar[r]^{\phi_1} \ar[d]^{\phi_3}    &  \pi^T \pi_* \mathcal{K}_M \ar[d]^{\phi_2} \\
      \mathcal{G} \ar[r] &  \mathcal{K}_M }
\end{xy}
\end{eqnarray*}
where $\phi_1$ is injective because $\pi_*$ and $\pi^T$ are left-exact,
and $\phi_2$ is injective by Lemma \ref{lem:ac1}.
Thus, $\phi_3$ is also injective. Note that we could not use Lemma \ref{lem:ac1} directly because $\mathcal{G}$
is not necessarily coherent. But injectivity of $\phi_3$ and $\mathcal{K}_X^s=\pi_* \mathcal{G}$
yield an injective morphism $\pi_* \pi^T \mathcal{K}_X^s \rightarrow \mathcal{K}_X^s$ 
(which is inverse to $\mathcal{K}_X^s \rightarrow \pi_* \pi^T \mathcal{K}_X^s$).
\end{proof}

For the sake of completeness, let us also include:

\begin{lem}\label{lem:ac4}
Let $X$ be a complex space and $i: \mathcal{F} \hookrightarrow \mathcal{G}$ an injective morphism 
between two coherent locally free sheaves of rank one over $X$.
Then there exists an effective Cartier divisor, $D\geq 0$, such that $i( \mathcal{F}) = \mathcal{G} \otimes \OO_X(-D)$.
In particular, $i$ is an isomorphism precisely on $X-|D|$.
\end{lem}

\begin{proof}
Let $\{X_\alpha\}_\alpha$ be a locally finite open cover of $X$ such that both, $\mathcal{F}$ and $\mathcal{G}$, are free over each $X_\alpha$.
So, there are trivializations
$$\phi_\alpha:  \mathcal{F}|_{X_\alpha} \overset{\sim}{\longrightarrow} \OO_{X_\alpha}\ ,\ \ 
\psi_\alpha: \mathcal{G}|_{X_\alpha} \overset{\sim}{\longrightarrow} \OO_{X_\alpha}\ ,$$
and for $X_{\alpha\beta}:=X_\alpha\cap X_\beta\neq \emptyset$, 
we have transition functions $F_{\beta\alpha}:= \phi_\beta\circ \phi_\alpha^{-1} \in \OO^*(X_{\alpha\beta})$
and $G_{\beta\alpha}:= \psi_\beta\circ \psi_\alpha^{-1} \in \OO^*(X_{\alpha\beta})$ satisfying the cocycle conditions.
In trivializations 
\begin{eqnarray*}
\psi_\alpha \circ i|_{X_\alpha} \circ \phi_\alpha^{-1}: \OO_{X_\alpha} \rightarrow \OO_{X_\alpha}
\end{eqnarray*}
is given by a holomorphic function $i_\alpha\in \OO(X_\alpha)$, vanishing nowhere identically,
with (unreduced) divisor $(i_\alpha)$. It is easy to see that
$G_{\beta\alpha} \cdot i_\alpha = i_\beta \cdot F_{\beta\alpha}$
on $X_{\alpha\beta}$, so that $i_\alpha / i_\beta = F_{\beta\alpha} / G_{\beta\alpha} \in \OO^*(X_{\alpha\beta})$.
Thus $D:=\{(X_\alpha,i_\alpha)\}_\alpha$ defines in fact an effective Cartier divisor with support $E$.

To see that $i(\mathcal{F})=\mathcal{G}\otimes\OO_X(-D)$, note that
$\mathcal{G}\otimes \OO_X(-D)$ is a coherent subsheaf of $\mathcal{G}$ because $\OO_X(-D)$ is a sheaf of ideals in $\OO_X$,
and that
$$\psi_\alpha\otimes 1: \left.\mathcal{G}\otimes \OO_X(-D)\right|_{X_\alpha} \overset{\sim}{\longrightarrow} \OO_{X_\alpha} \otimes \OO_{X_\alpha}(-(i_\alpha)).$$
\end{proof}

%\bigskip
%\newpage

\end{document}